\documentclass[11pt,a4paper]{article}
\usepackage{graphicx}
\usepackage{amsmath,amsfonts,amssymb,amsthm}
\title{Rotation covariant local tensor valuations on convex bodies}
\author{Daniel Hug and Rolf Schneider}
\date{}
\newcommand{\R}{{\mathbb R}}
\newcommand{\Sn}{{\mathbb S}^{n-1}}
\newcommand{\Kn}{{\mathcal K}^n}
\newcommand{\Pn}{{\mathcal P}^n}
\newcommand{\cP}{{\mathcal P}}
\newcommand{\Rn}{{\mathbb R}^n}
\newcommand{\N}{{\mathbb N}}

\newcommand{\cH}{\mathcal{H}}
\newcommand{\Ha}{\mathcal{H}}
\newcommand{\bS}{\mathbb{S}}
\newcommand{\T}{\mathbb{T}}
\newcommand{\B}{\mathcal{B}}
\newcommand{\D}{{\rm d}}

\newcommand{\F}{{\mathcal F}}
\newcommand{\rO}{{\rm O}(n)}
\newcommand{\SO}{{\rm SO}(n)}
\newcommand{\3}{{\rm O}(3)}
\newcommand{\4}{{\rm SO}(3)}
\DeclareMathOperator{\Nor}{Nor}
\DeclareMathOperator{\sgn}{sgn}
\newcommand{\fed}{\,\rule{.1mm}{.20cm}\rule{.20cm}{.1mm}\,}

\newtheorem{lemma}{Lemma}
\newtheorem{theorem}{Theorem}
\newtheorem{proposition}{Proposition}

\begin{document}

\maketitle

\begin{abstract}
For valuations on convex bodies in Euclidean spaces, there is by now a long series of characterization and classification theorems. The classical template is Hadwiger's theorem, saying that every rigid motion invariant, continuous, real-valued valuation on convex bodies in $\Rn$ is a linear combination of the intrinsic volumes. For tensor-valued valuations, under the assumptions of isometry covariance and continuity, there is a similar classification theorem, due to Alesker. Also for the local extensions of the intrinsic volumes, the support, curvature and area measures, there are analogous characterization results, with continuity replaced by weak continuity, and involving an additional assumption of local determination. The present authors have recently obtained a corresponding characterization result for local tensor valuations, or tensor-valued support measures (generalized curvature measures), of convex bodies in $\Rn$. The covariance assumed there was with respect to the group ${\rm O}(n)$ of orthogonal transformations. This was suggested by Alesker's observation, according to which in dimensions $n> 2$, the weaker assumption of ${\rm SO}(n)$ covariance does not yield more tensor valuations. However, for tensor-valued support measures, the distinction between proper and improper rotations does make a difference. The present paper considers, therefore, the local tensor valuations sharing the previously assumed properties, but with ${\rm O}(n)$ covariance replaced by ${\rm SO}(n)$ covariance, and provides a complete classification. New tensor-valued support measures appear only in dimensions two and three.

\medskip

{\em Key words and phrases:} Valuation; Minkowski tensor; local tensor valuation; tensor-valued support measures; rotation covariance; classification theorem

\medskip

{\em Mathematics subject classification:}  52A20, 52B45
\end{abstract}

\section{Introduction}\label{sec1}

The ultimate goal of this paper is a classification result for continuous, rotation covariant valuations on convex bodies with values in a space of tensor-valued measures. Well-known characterization results of Hadwiger and Alesker are landmarks in this line of research, which we now briefly explain. 

A {\em valuation} on the space $\Kn$ of convex bodies in Euclidean space $\Rn$ is a mapping $\varphi$ from $\Kn$ into some abelian group satisfying 
$$\varphi(K\cup M)+\varphi(K\cap M)=\varphi(K)+\varphi (M)$$ 
for all $K,M\in \Kn$ with $K\cup M\in \Kn$. Of geometric interest are mainly those valuations which have a simple behaviour under some transformation group of $\Rn$ and which have certain continuity properties; in recent investigations, even suitable smoothness assumptions play an important role. Hadwiger's celebrated characterization theorem says that the vector space of real-valued valuations on $\Kn$ which are invariant under rigid motions and are continuous with respect to the Hausdorff metric, is spanned by the intrinsic volumes  (\cite{Had52}, reproduced in \cite[6.1.10]{Had57}). A similar result for $\Rn$-valued valuations was proved by Hadwiger and Schneider \cite{HS71}, based on a characterization result in \cite{Sch72}. A first systematic study of valuations with values in a space of symmetric tensors on $\Rn$ was begun by McMullen \cite{McM97}. A corresponding classification theorem was proved by Alesker \cite{Ale99b}, based on his earlier results in \cite{Ale99a}. The isometry covariance underlying this result includes a certain polynomial behaviour under  translations and covariance with respect to the orthogonal group ${\rm O}(n)$ of $\Rn$.

Once the tensor-valued generalizations of the intrinsic volumes, also known as Minkowski tensors, had been introduced, they were investigated and applied under various different viewpoints; see, for example, the Lecture Notes \cite{KJ16}. So they appeared in integral geometry (\cite{BH14}, \cite{HSS08b}, \cite{Sch00}, \cite{SS06}), in stochastic geometry, stereology and image analysis (\cite{HHKM14}, \cite{HKS15}, \cite{KKH14}, \cite{SS02}), and have been applied to different topics in physics (\cite{BDMW02}, \cite{MKSM13}, \cite{SchT10}, \cite{SchT11}, \cite{SchT13}). For the latter applications, dimension three, which is in the focus of the present paper, is of particular interest.

The classical intrinsic volumes can be localized: they are just the total measures of the curvature measures, area measures, or support measures associated with a convex body. Here a curvature measure of a convex body is concentrated on (Borels sets of, in each case) boundary points, an area measure on unit normal vectors, and a support measure on support elements, that is, pairs of boundary point and normal vector at the point. In dependence on the convex bodies, these measures are valuations and are weakly continuous. Further, they have certain properties of covariance with respect to motion groups and of local determination. Classification theorems assuming these properties were proved for area measures in \cite{Sch75} and for curvature measures in \cite{Sch78}. In the case of support measures on convex polytopes, it was first observed by Glasauer \cite[Lem.~1.3]{Gla97} that the properties of rigid motion equivariance and local determination are sufficient for a characterization theorem. Thus, no continuity assumptions are required, and the valuation property is a consequence.

Also the tensor valuations can be localized, which leads to the local Minkowski tensors or, as Saienko \cite{Sai16} suggests to call them, tensor-valued curvature measures, or, in our case, tensor-valued support measures. When restricted to polytopes, they can be completely classified under the sole assumptions of isometry covariance and local determination. This was essentially done in \cite{Sch13} and later slightly strengthened, 
see Theorem 2.2 in \cite{HS14}. The question, which of these local tensor valuations have weakly continuous extensions to all convex bodies, was completely settled in \cite{HS14}.

The isometry covariance that is assumed in the previous classification of local tensor valuations, as well as in Alesker's \cite{Ale99b} characterization theorem, comprises covariance with respect to the group $\rO$ of orthogonal transformations. If one assumes only covariance with respect to the group $\SO$ of proper rotations (orientation preserving orthogonal transformations), then Alesker \cite[Sec. 4]{Ale99b} pointed out that in his classification theorem one gets more tensor valuations if $n=2$, but not if $n\ge 3$. Therefore, it came as a surprise when Saienko, in his work on smooth tensor-valued curvature measures, discovered that in dimension three there are such valuations which are covariant with respect to $\4$, but not with respect to $\3$. Why this is consistent with Alesker's assertion, is explained in Section \ref{sec8}.

Saienko's discovery in the smooth case was a motivation to revisit the classifications in \cite{Sch13} and \cite{HS14} and to replace the assumption of $\rO$-covariance by that of $\SO$-covariance. In \cite{HS16}, where local tensor valuations on polytopes without any continuity assumption are considered, it was found that the classification obtained in \cite{Sch13} and \cite[Thm. 2.2]{HS14} remains unchanged in dimensions $n \ge 4$, but that new $\SO$ covariant local tensor valuations appear for $n=2$ and $n=3$. They were completely classified in \cite{HS16}. The purpose of the present paper is now to find out which of these have a weakly continuous extension to all convex bodies, and to extend the classification theorem correspondingly. The main result, whose precise formulation requires some more preparations, is Theorem \ref{Thm4.1} in Section \ref{sec4}. 

The proof of our main classification result is completed in Section \ref{sec7}. Sections \ref{sec8} and \ref{sec9} are then devoted to the tensor valuations that are defined by the total measures of the ${\rm SO}(n)$ covariant but not ${\rm O}(n)$ covariant local tensor valuations that exist in dimensions $n=2$ and $n=3$. For $n=3$ we show in Section \ref{sec8} that they are zero, and in Section {\ref{sec9} we determine for $n=2$ all linear dependences between them.

\section{Notation and Preliminaries}\label{sec2}

We introduce the basic notations for a general dimension $n\ge 2$. The $n$-dimensional real vector space $\Rn$ is equipped with its standard scalar product $\langle\cdot\,,\cdot\rangle$ and the induced norm $\|\cdot\|$. We also assume that $\Rn$ is endowed with a fixed orientation. The $k$-dimensional Hausdorff measure on $\Rn$ is denoted by $\cH^k$. We need the unit sphere $\Sn$ and the product $\Sigma^n=\Rn\times\Sn$, both with their standard topologies. The constant $\cH^{n-1}(\Sn)=\omega_n= 2\pi^{n/2}/\Gamma(n/2)$ appears occasionally. If $L\subset\Rn$ is a linear subspace, we write ${\mathbb S}_L= \Sn\cap L$. By $\rO$ we denote the orthogonal group of $\Rn$, that is, the group of linear transformations preserving the scalar product, and $\SO$ is the subgroup of proper rotations, which preserve also the orientation. The Grassmannian of $k$-dimensional linear subspaces of $\Rn$ is denoted by $G(n,k)$.

The set $\Kn$ of convex bodies (nonempty, compact, convex subsets) in $\Rn$ is equipped with the Hausdorff metric and its induced topology. The subset of polytopes is denoted by $\Pn$. For a polytope $P$, the set of its $k$-dimensional faces is denoted by $\F_k(P)$, for $k=0,\dots,\dim P$. For a face $F$, we write $L(F)={\rm lin}(F-F)$; this is the linear subspace that is parallel to the affine hull of $F$ and is called the {\em direction space} of $F$. The normal cone of $P$ at its face $F$ is denoted by $N(P,F)$, and $\nu(P,F)= N(P,F)\cap\Sn\subset {\mathbb S}_{L(F)^\perp}$ is the set of outer unit normal vectors of $P$ at $F$. The generalized normal bundle (or normal cycle) of $P$ is the subset ${\rm Nor}\,P\subset\Sigma^n$ consisting of all pairs $(x,u)$ such that $x$ is a boundary point of $P$ and $u$ is an outer unit normal vector of $P$ at $x$. The same notation and terminology is used for general convex bodies.  

We recall the conventions on tensors that were used in \cite{Sch13}, \cite{HS14}, \cite{HS16}. For $p\in\N_0$, we denote by $\T^p$ the real vector space of symmetric tensors of rank $p$ on $\R^n$. The scalar product $\langle\cdot\,,\cdot\rangle$ of $\R^n$ is used to identify $\R^n$ with its dual space, so that each vector $a\in\R^n$ is identified with the linear functional $x\mapsto \langle a,x\rangle$, $x\in\R^n$. Thus, $\T^1$ is identified with $\R^n$ (and $\T^0$ with $\R$), and for $p\ge 1$, each tensor $T\in\T^p$ is considered as a symmetric $p$-linear functional on $\R^n$. The symmetric tensor product $a\odot b$ is abbreviated by $ab$, and for $x\in\R^n$, the $r$-fold symmetric tensor product $x\odot\dots\odot x$ is denoted by $x^r$, with $x^0:=1$. 

The {\em metric tensor} $Q$ on $\R^n$ is defined by $Q(x,y):=\langle x,y\rangle$ for $x,y\in \R^n$. For a subspace $L\in G(n,k)$, we denote by $\T^p(L)$ the space of symmetric $p$-tensors on $L$. The tensor $Q_L$ on $\Rn$ is defined by
$$ Q_L(a,b):= \langle \pi_La,\pi_Lb\rangle\quad\mbox{for } a,b\in\Rn,$$
where $\pi_L:\Rn\to L$ denotes the orthogonal projection.

Generally for a topological space $X$, we denote by $\B(X)$ the $\sigma$-algebra of its Borel sets. 

In the following, we are concerned with mappings
$$ \Gamma: \Kn\times\B(\Sigma^n)\to \T^p.$$
We say, briefly, that such a mapping is a {\em valuation} if $\Gamma(\cdot,\eta)$ is a valuation for each $\eta\in\B(\Sigma^n)$. If $\Gamma(K,\cdot)$ is a $\T^p$-valued measure for each $K\in\Kn$, then $\Gamma$ is called {\em weakly continuous} if
$$ \lim_{i\to\infty} \int_{\Sigma^n} f\,\D\Gamma(K_i,\cdot)= \int_{\Sigma^n} f\,\D\Gamma(K,\cdot)$$
holds for each sequence $(K_i)_{i\in{\mathbb N}}$ of convex bodies in $\Kn$ with limit $K$ and for each continuous function $f:\Sigma^n\to\R$. In the following, for $\eta\subset \Sigma^n$ we write $\eta+t:= \{(x+t,u):(x,u)\in\eta\}$ for $t\in\Rn$, $\lambda\eta:= \{(\lambda x,u): (x,u)\in\eta\}$ for $\lambda\ge 0$, and $\vartheta\eta:=\{(\vartheta x,\vartheta u): (x,u)\in\eta\}$ for $\vartheta\in {\rm O}(n)$. The mapping $\Gamma$ is called {\em translation covariant} of degree $q\le p$ if
\begin{equation}\label{2.1} 
\Gamma(K+t,\eta+t) =\sum_{j=0}^q \Gamma_{p-j}(K,\eta)\frac{t^j}{j!}
\end{equation}
with tensors $\Gamma_{p-j}(K,\eta)\in\T^{p-j}$, for $K\in\Kn$, $\eta\in\B(\Sigma^n)$, and $t\in\Rn$. Here $\Gamma_p=\Gamma$. If $\Gamma$ is translation covariant of degree zero, it is called {\em translation invariant}, and $\Gamma$ is just called {\em translation covariant} if it is translation covariant of some degree $q\le p$. The mapping $\Gamma$ is called ${\rm O}(n)$ {\em covariant} if $\Gamma(\vartheta K,\vartheta \eta)= \vartheta \Gamma(K,\eta)$ for $K\in\Kn$, $\eta\in\B(\Sigma^n)$, $\vartheta\in{\rm O}(n)$. Here the operation of  ${\rm O}(n)$ on $\T^p$ is defined by $(\vartheta T)(x_1,\dots,x_p):= T(\vartheta^{-1}x_1,\dots,\vartheta^{-1}x_p)$ for $x_1,\dots,x_p\in\R^n$ and $\vartheta\in {\rm O}(n)$. Similarly, ${\rm SO}(n)$-covariance is defined. We say that the mapping $\Gamma$ is {\em locally defined} if $\eta\cap {\rm Nor}\,K = \eta'\cap {\rm Nor}\,K'$ with $K,K'\in\Kn$ and $\eta,\eta'\in\B(\Sigma^n)$ implies $\Gamma(K,\eta)=\Gamma(K',\eta')$. The mapping $\Gamma$ is {\em homogeneous} of degree $k$ if $\Gamma(\lambda K,\lambda\eta)=\lambda^k\Gamma(K,\eta)$. Corresponding definitions are used if $\Kn$ in the definition of $\Gamma$ is replaced by $\Pn$.

The following types of tensor-valued support measures must be distinguished. The {\em local Minkowski tensors} are defined by
\begin{equation}\label{A} 
\phi_k^{r,s}(K,\eta) =  \frac{\omega_{n-k}}{r!s!\,\omega_{n-k+s}}\int_\eta x^ru^s\,\Lambda_k(K,\D(x,u))
\end{equation}
for $K\in\Kn$, $\eta\in\B(\Sigma^n)$, $r,s\in{\mathbb N}_0$, $k\in\{0,\dots,n-1\}$, where $\Lambda_0,\dots,\Lambda_{n-1}$ are the support measures (see \cite[Sec. 2]{HS14} for explanations). For a polytope $P\in\Pn$, there is a more explicit expression, namely
$$ \phi_k^{r,s}(P,\eta)= \frac{1}{r!s!\,\omega_{n-k+s}}\sum_{F\in\F_k(P)} \int_F \int_{\nu(P,F)} {\bf 1}_\eta(x,u)x^ru^s\,\cH^{n-k-1}(\D u)\,\cH^k(\D x).$$
The {\em generalized local Minkowski tensors} of a polytope $P\in\Pn$ were in \cite{HS14} defined by
\begin{equation}\label{4.1}  
\phi_k^{r,s,j}(P,\eta) = \frac{1}{r!s!\,\omega_{n-k+s}} \sum_{F\in\F_k(P)} Q_{L(F)}^{\hspace*{1pt}j}\int_F \int_{\nu(P,F)} {\bf 1}_\eta(x,u) x^r u^s\, \Ha^{n-k-1}(\D u)\,\Ha^k(\D x)
\end{equation}
for $\eta\in\B(\Sigma^n)$, $k\in\{0,\dots,n-1\}$, $r,s\in\N_0$, and for $j\in{\mathbb N}_0$ if $k>0$, but only $j=0$ if $k=0$. 

The mapping defined by $\Gamma(P,\eta)= \phi_k^{r,s,j}(P,\eta)$, for fixed $k,r,s,j$, has the following properties. It is a valuation. For each $P\in\Pn$, $\Gamma(P,\cdot)$ is a $\T^p$-valued measure, with $p=2j+r+s$. The mapping $\Gamma$ is translation covariant, ${\rm O}(n)$ covariant, and locally defined. These properties are not changed (except that the rank must be adjusted) if $\Gamma$ is multiplied (symmetrically) by a power of the metric tensor.

It was proved in \cite{HS14} that the mapping $ \phi_k^{r,s,j}$ has a weakly continuous extension (denoted by the same symbol) to $\Kn\times\B(\Sigma^n)$ if $k=n-1$ or if $j\in\{0,1\}$, but not if $j\ge 2$ and $1\le k\le n-2$. Then, the following classification result was obtained (\cite[Thm. 2.3]{HS14}).

\begin{theorem}\label{Thm2.1}
For $p\in\N_0$, let $T_p(\Kn)$ denote the real vector space of all mappings $\Gamma:\Kn\times\B(\Sigma^n)\to\T^p$ with the following properties.\\
$\rm (a)$ $\Gamma(K,\cdot)$ is a $\T^p$-valued measure, for each $K\in\Kn$,\\ 
$\rm (b)$ $\Gamma$ is translation covariant and $\rO$ covariant,\\
$\rm (c)$ $\Gamma$ is locally defined,\\
$\rm (d)$ $\Gamma$ is weakly continuous.

Then a basis of $T_p(\Kn)$ is given by the mappings $Q^m\phi^{r,s,j}_k$, where $m,r,s\in\N_0$ and $j\in\{0,1\}$ satisfy $2m+2j+r+s=p$, where $k\in\{0,\dots,n-1\}$, and where $j=0$ if $k\in\{0,n-1\}$.
\end{theorem}

Aiming at replacing $\rO$-covariance by $\SO$-covariance, it turned out in \cite{HS16} that we had to introduce further local tensor valuations in dimensions two and three. For $P\in\cP^3$ and $\eta\in\B(\Sigma^3)$, let
\begin{equation}\label{2.2} 
\widetilde\phi^{r,s,j}(P,\eta) = \sum_{F\in{\mathcal F}_1(P)}  v_F^{2j+1}\int_F\int_{\nu(P,F)} {\bf 1}_\eta(x,u)x^r(v_F\times u) u^s\,\cH^1(\D u)\, \cH^1(\D x),
\end{equation}
where $r,s,j\in{\mathbb N}_0$. The vector $v_F$ is one of the two unit vectors (arbitrarily chosen) parallel to the edge $F$, and $v_F\times u$ denotes the vector product of the vectors $v_F$ and $u$. The right side of (\ref{2.2}) is independent of the choice of the vector $v_F$.

For $n=2$ and for $u\in {\mathbb S}^1$, let $\overline u\in {\mathbb S}^1$ be the unique vector for which $(u,\overline u)$ is a positively oriented orthonormal basis of $\R^2$. For $P\in{\mathcal P}^2$, $k\in\{0,1\}$ and $\eta \in \B(\Sigma^2)$ we define
\begin{equation}\label{2.3} 
\widetilde \phi_k^{r,s}(P,\eta)= \sum_{F\in\F_k(P)} \int_F \int_{\nu(P,F)} {\bf 1}_\eta(x,u) x^r \overline u\, u^s\,\Ha^{1-k}(\D u)\,\Ha^k(\D x).
\end{equation}

The mapping defined by $\Gamma(P,\eta):= \widetilde\phi^{r,s,j}(P,\eta)$, for fixed $r,s,j$ if $n=3$, and by  $\Gamma(P,\eta):= \widetilde\phi^{r,s}_k(P,\eta)$ for fixed $r,s,k$ if $n=2$, has the following properties. It is a valuation. For each $P\in\cP^n$, $\Gamma(P,\cdot)$ is a $\T^p$-valued measure, for suitable $p$. $\Gamma$ is translation covariant, ${\rm SO}(n)$ covariant, and locally defined. If $\vartheta\in {\rm O}(n)$ changes the orientation, then $\Gamma(\vartheta P,\vartheta \eta)= -\vartheta \Gamma(P,\eta)$.

The following result was proved in \cite{HS16}.

\begin{theorem}\label{Thm2.2}
For $p\in\N_0$, let $\widetilde T_p(\Pn)$ denote the real vector space of all mappings $\Gamma:\Pn\times\B(\Sigma^n)\to\T^p$ with the following properties.\\[1mm]
$\rm (a)$ $\Gamma(P,\cdot)$ is a $\T^p$-valued measure, for each $P\in\Pn$,\\[1mm]
$\rm (b)$ $\Gamma$ is translation covariant and ${\rm SO}(n)$ covariant,\\[1mm]
$\rm (c)$ $\Gamma$ is locally defined.

Then a basis of $\widetilde T_p(\Pn)$ is given by the mappings $Q^m\phi^{r,s,j}_k$, where $m,r,s,j\in\N_0$ satisfy $2m+2j+r+s=p$, where $k\in\{0,\dots,n-1\}$, and where $j=0$ if $k\in\{0,n-1\}$, together with\\[1mm]
$\bullet$ if $n\ge 4$, no more mappings,\\[1mm]
$\bullet$ if $n=3$, the mappings $Q^m\widetilde \phi^{r,s,j}$, where $m,r, s,j\in\N_0$ satisfy $2m+2j+r+s+2=p$,\\[1mm]
$\bullet$ if $n=2$, the mappings $Q^m\widetilde \phi_k^{r,s}$, where $m,r,s\in\N_0$ satisfy $2m+r+s+1=p$ and where \hspace*{6pt}  $k\in\{0,1\}$.
\end{theorem}

In order to extend this theorem from $\cP^n$ to $\Kn$, under additional continuity assumptions, we have to investigate which of the mappings $\widetilde \phi^{r,s,j}$ and  $\widetilde \phi_k^{r,s}$ have weakly continuous extensions from $\cP^n$ to $\Kn$, for $n=3$ respectively $n=2$. The next section provides weakly continuous extensions for $n=2$ and for $n=3$, $j=0$. The rest of the paper will then reveal that there are no such extensions in the remaining cases.

\section{Weakly Continuous Extensions}\label{sec3}

The case $n=2$ is easily settled, because for $K\in {\mathcal K}^2$ we can define
\begin{equation}\label{B}  
\widetilde\phi_k^{r,s}(K,\eta) =\int_\eta x^r\overline u u^s\,\Theta_k(K,\D(x,u)),
\end{equation}
with the support measure $\Theta_k$ defined in \cite[Sec.~4.2]{Sch14}. For $P\in{\mathcal P}^2$, this is consistent with (\ref{2.3}) (by \cite[(4.3)]{Sch14}). Since the support measures are weakly continuous, we obtain that $\widetilde\phi_k^{r,s}$ thus defined is weakly continuous on ${\mathcal K}^2$.

Now we turn to the case $n=3$. 
The construction of a weakly continuous extension to all convex bodies of the functional
 $\widetilde\phi^{r,s,0}$, defined so far on $\mathcal{P}^3$, requires some preparations. The basic strategy is the same 
as in \cite[Sec.~4]{HS14}. The main task is to define suitable smooth tensor-valued differential forms $\psi^{r,s}$ of degree $n-1=2$ 
on $\R^6=\R^3\times\R^3$ taking values in $\mathbb{T}^{r+s+2}$. For a convex body $K\in\mathcal{K}^3$, the normal cycle $T_K$ 
is defined by 
$$
T_K:=\left(\mathcal{H}^{2}\fed\Nor K\right)\wedge a_K,
$$
where 
$$
a_K(x,u):=a_1(x,u)\wedge  a_{2}(x,u)
$$
is a $2$-vector in $\R^6$ which determines an orientation of the two-dimensional approximate tangent space 
$\textrm{Tan}^{2}(\mathcal{H}^{2}\fed \Nor K,(x,u))$ of the generalized normal bundle $\Nor K$ of $K$, 
for $\mathcal{H}^{2}$-almost all $(x,u)\in\Nor K$. We refer to Federer's book \cite{Fed69} for basic terminology and results 
of geometric measure theory, and to \cite[Sec.~4]{HS14} for further details and references. In particular, we choose an orthonormal 
basis $(b_1(x,u),b_2(x,u))$ in the orthogonal complement $u^\perp$ of $u\in\mathbb{S}^2$ such that $(b_1(x,u),b_2(x,u),u)$ is a positively oriented orthonormal basis of $\R^3$ and such that
$$
a_i(x,u):=\left( \frac{1}{\sqrt{1+{k_i(x,u)}^2}}\,b_i(x,u),\frac{k_i(x,u)}{\sqrt{1+{k_i(x,u)}^2}}\,b_i(x,u)\right),\quad i=1,2,
$$
with $k_i(x,u)\in[0,\infty]$. Note that if $P\in\mathcal{P}^3$ , $F\in\mathcal{F}_1(P)$, $x\in \text{relint}\, F$ 
and $u\in\text{relint}\, \nu(P,F)$, then we can choose $b_1(x,u)=v_F$ with $k_1(x,u)=0$ and 
$b_2(x,u)=u\times b_1(x,u)$ with $k_2(x,u)=\infty$, where the usual convention $1/\sqrt{1+k_i(x,u)^2}=0$ and 
$k_i(x,u)/\sqrt{1+k_i(x,u)^2}=1$ for $k_i(x,u)=\infty$ is used in the following. 

In a first step, for $(x,u)\in\R^6$, $\mathbf{v}=(v_1,\ldots,v_{r+s+2})\in (\R^3)^{r+s+2}$ and $\xi_1,\xi_2\in\R^6=\R^3\times\R^3$ 
we define
\begin{eqnarray*}
 \widetilde{\psi}^{r,s}(x,u;\mathbf{v};\xi_1,\xi_{2})&:=&-\, x^r(v_1,\ldots,v_r)u^s(v_{r+1},\ldots,v_{r+s})\\
& &  \times\sum_{\sigma\in {\mathcal S}(2)}\sgn(\sigma)\left\langle v_{r+s+1},\Pi_1\xi_{\sigma(1)}\right\rangle 
\left\langle v_{r+s+2},\Pi_2\xi_{\sigma(2)}\right\rangle.
\end{eqnarray*}
Here and below, ${\mathcal S}(n)$ denotes the group of permutations of the set $\{1,\dots,n\}$, and 
$\Pi_1,\Pi_2:\R^3\times\R^3\to\R^3$ are the projections $\Pi_1(a,b):=a$ and 
$\Pi_2(a,b):=b$, so that 
$$
\widetilde{\psi}^{r,s}(x,u;\mathbf{v};\cdot)\in \text{$\bigwedge$}^{2}\,\R^6.
$$
In a second step,  we symmetrize
$$ {\psi}^{r,s}(x,u;\mathbf{v};\xi_1,\xi_{2})\\
 :=\frac{1}{(r+s+2)!}\sum_{\tau\in {\mathcal S}(r+s+2)}
\widetilde{\psi}^{r,s}(x,u;v_{\tau(1)},\ldots,v_{\tau(r+s+2)};\xi_1,\xi_{2}).
$$
Hence, for any $r,s\in\N_0$ we have defined smooth tensor-valued differential forms $\psi^{r,s}\in\mathcal{E}^2(\R^6,\mathbb{T}^{r+s+2})$. 

The next lemma shows that these differential forms are suitably defined, for the construction of an extension of 
the new tensor valuations $\widetilde{\phi}^{r,s,0}$ to all convex bodies.

\begin{lemma}\label{lem3.1}
If $P\in\mathcal{P}^3$ and $\eta\in\mathcal{B}(\Sigma^3)$, then 
$$
T_P\left(\mathbf{1}_\eta \psi^{r,s}\right)=\widetilde{\phi}^{r,s,0}(P,\eta).
$$
\end{lemma}

\begin{proof}
For a given polytope $P\in\mathcal{P}^3$ and $\eta\in\mathcal{B}(\Sigma^3)$, we show that 
\begin{eqnarray*}
T_P\left(\mathbf{1}_\eta\widetilde\psi^{r,s}(\cdot;\mathbf{v};\cdot)\right)&=&\sum_{F\in\mathcal{F}_1(P)}\int_{\eta\cap (F\times\nu(P,F))}
x^r(v_1,\ldots,v_r)u^s(v_{r+1},\ldots,v_{r+s})\\
&&\times \langle v_F,v_{r+s+1}\rangle\langle v_{r+s+2},v_F\times u\rangle\, \mathcal{H}^2(\D(x,u))
\end{eqnarray*}
for all $\mathbf{v}=(v_1,\ldots,v_{r+s+2})\in(\R^3)^{r+s+2}$. Subsequent symmetrization then yields the assertion of the lemma. 

Using the disjoint decomposition
$$
\eta\cap \Nor P=\bigcup_{j=0}^{2}\bigcup_{F\in\mathcal{F}_j(P)} \eta\cap(\textrm{relint }F\times\nu(P,F)),
$$
we obtain
\begin{eqnarray*}
 T_P\left(\mathbf{1}_{\eta} \widetilde{\psi}^{r,s}(\cdot\,;\mathbf{v};\cdot)\right)
&  =&\int_{\eta\cap\Nor P} \left\langle a_P(x,u), \widetilde{\psi}^{r,s}(x,u;\mathbf{v};\cdot)\right\rangle \, \mathcal{H}^{2}(\D (x,u)) \\
&=&-\sum_{j=0}^2 \sum_{F\in\mathcal{F}_j(P)}\int_{\eta\cap(F\times \nu(P,F)) }x^r(v_1,\ldots,v_r)u^s(v_{r+1},\ldots,v_{r+s})\\
&&\times \sum_{\sigma \in {\mathcal S}(2)}\sgn(\sigma)\left\langle v_{r+s+1},\frac{1}{\sqrt{1+k_{\sigma(1)}(x,u)^2}}\,b_{\sigma(1)}(x,u)\right\rangle \\
&&\times \left\langle v_{r+s+2},\frac{k_{\sigma(2)}(x,u)}{\sqrt{1+k_{\sigma(2)}(x,u)^2}}\,b_{\sigma(2)}(x,u)\right\rangle 
 \mathcal{H}^{2}(\D (x,u))\\
&=& - \sum_{F\in\mathcal{F}_1(P)}\int_{\eta\cap(F\times \nu(P,F)) }x^r(v_1,\ldots,v_r)u^s(v_{r+1},\ldots,v_{r+s})\\
&&\times \left\langle v_{r+s+1},1\cdot b_{1}(x,u)\right\rangle \left\langle v_{r+s+2},1\cdot b_{2}(x,u)\right\rangle 
 \mathcal{H}^{2}(\D (x,u))\\
&=&\sum_{F\in\mathcal{F}_1(P)}\int_{\eta\cap(F\times \nu(P,F)) }x^r(v_1,\ldots,v_r)u^s(v_{r+1},\ldots,v_{r+s})\\
&&\times 
\langle v_{r+s+1},v_F \rangle\langle  v_{r+s+2}, -u\times v_F \rangle\, \mathcal{H}^2(\D(x,u)),
\end{eqnarray*}
where we used that for $F\in\mathcal{F}_j(P)$ and $\mathcal{H}^2$ almost all $(x,u)\in F\times \nu(P,F)$, we have 
$k_1(x,u)=k_2(x,u))=\infty$ if $j=0$ and $k_1(x,u)=k_2(x,u))=0$ if $j=2$; moreover, for $j=1$, we can choose $b_1(x,u)=v_F$ with 
$k_1(x,u)=0$ and $b_2(x,u)=u\times b_1(x,u)$ with $k_2(x,u)=\infty$. This completes the arguments.
\end{proof}

Combining Lemma \ref{lem3.1}, \cite[Lem. 4.2]{HS14}, the argument on page 1542 in \cite{HS14} and in the proof 
of Theorem 4.1 of \cite{HS14}, we obtain the following result. 

\begin{theorem}\label{theorem 3.1}
The map $\mathcal{K}^3\times\mathcal{B}(\Sigma^3)\to\T^{r+s+2}$, $K\mapsto T_K\left(\mathbf{1}_\eta\psi^{r,s}\right)$,  is the weakly 
continuous extension of the map $\mathcal{P}^3\times\mathcal{B}(\Sigma^3)\to\T^{r+s+2}$, $(P,\eta)\mapsto\widetilde\phi^{r,s,0}(P,\eta)$. 
The extension is a tensor-valued measure which is a translation covariant and ${\rm SO}(3)$ covariant, locally defined and weakly continuous 
valuation. 
\end{theorem}

In the following, we denote the weakly continuous extension of $\widetilde\phi^{r,s,0}$ to all convex bodies again by $\widetilde\phi^{r,s,0}$. 
For general convex bodies, we write 
$$
\mathbb{K}(x,u):=\sqrt{1+k_1(x,u)^2}\sqrt{1+k_2(x,u)^2},
$$
where the dependence on $K$ is not indicated. 
Omitting also the arguments of $k_i,b_i$ and $\mathbb{K}$, we obtain
\begin{align*}
\widetilde\phi^{r,s,0}(K,\eta)&=\int_{\eta\cap\Nor K}-x^r u^s \left[\frac{1}{\sqrt{1+k_1^2}}\,b_1\frac{k_2}{\sqrt{1+k_2^2}}\,b_2-
\frac{1}{\sqrt{1+k_2^2}}\,b_2\frac{k_1}{\sqrt{1+k_1^2}}\,b_1\right]\mathcal{H}^2(\D(x,u))\\
&=\int_{\eta\cap\Nor K}x^r u^s \frac{k_1b_1  
b_2-k_2b_1b_2}{\mathbb{K}}\, \mathcal{H}^2(\D(x,u))\\
&=\int_{\eta\cap\Nor K}x^r u^s
\frac{k_1b_1( u\times 
b_1)+k_2b_2(u\times b_2)}{\mathbb{K}}\, \mathcal{H}^2(\D(x,u))
\end{align*}
or
$$
\widetilde\phi^{r,s,0}(K,\eta)=\int_{\eta\cap\Nor K}x^r u^s
\frac{k_1b_2( b_2\times 
u)+k_2b_1( b_1\times u)}{\mathbb{K}}\, \mathcal{H}^2(\D(x,u)),
$$
where we used that $b_2=u\times b_1$ and $-b_1=u\times b_2$.

\section{The Classification Theorem}\label{sec4}

The following extension of Theorem \ref{Thm2.1} is the main result of this paper. 

\begin{theorem}\label{Thm4.1}
For $p\in\N_0$, let $\widetilde T_p(\Kn)$ denote the real vector space of all mappings $\Gamma:\Kn\times\B(\Sigma^n)\to\T^p$ with the following properties.\\[1mm]
$\rm (a)$ $\Gamma(K,\cdot)$ is a $\T^p$-valued measure, for each $K\in\Kn$,\\[1mm]
$\rm (b)$ $\Gamma$ is translation covariant and ${\rm SO}(n)$ covariant,\\[1mm]
$\rm (c)$ $\Gamma$ is locally defined,\\[1mm]
$\rm (d)$ $\Gamma$ is weakly continuous.

Then a basis of $\widetilde T_p(\Kn)$ is given by the mappings $Q^m\phi^{r,s,j}_k$, where $m,r,s\in\N_0$ and $j\in \{0,1\}$ satisfy $2m+2j+r+s=p$, where $k\in\{0,\dots,n-1\}$, and where $j=0$ if $k\in\{0,n-1\}$, together with\\[1mm]
$\bullet$ if $n\ge 4$, no more mappings,\\[1mm]
$\bullet$ if $n=3$, the mappings $Q^m\widetilde \phi^{r,s,0}$, where $m,r,s\in\N_0$ satisfy $2m+r+s+2=p$,\\[1mm]
$\bullet$ if $n=2$, the mappings $Q^m\widetilde \phi_k^{r,s}$, where $m,r,s\in\N_0$ satisfy $2m+r+s+1=p$ and where \hspace*{6pt}  $k\in\{0,1\}$.
\end{theorem}

For $n\not= 3$, this follows from the previous results. In fact, suppose that $\Gamma:\Kn\times\B(\Sigma^n)\to\T^p$ satisfies (a)--(d). First let $n\ge 4$. Then Theorem \ref{Thm2.2} tells us that the restriction of $\Gamma$ to $\Pn\times\B(\Sigma^n)$ is a linear combination of certain mappings $Q^m\phi_k^{r,s,j}$, restricted to $\Pn\times\B(\Sigma^n)$. Since on $\Kn\times\B(\Sigma^n)$ the mappings $Q^m\phi_k^{r,s,j}$, as well as $\Gamma$, are weakly continuous, the linear combination extends to general convex bodies. If $n=2$, then it follows from Theorem \ref{Thm2.2} that the restriction of $\Gamma$ to $\mathcal{P}^2\times\B(\Sigma^2)$ is a linear combination of certain mappings $Q^m\phi_k^{r,s,j}$ and certain mappings $Q^m\widetilde\phi_k^{r,s}$, restricted to $\mathcal{P}^2\times\B(\Sigma^2)$. Again by weak continuity, this linear combination extends to ${\mathcal K}^2$. The linear independence
holds already for the restrictions.

Thus, it remains to prove Theorem \ref{Thm4.1} for $n=3$. By an argument already used in \cite[pp. 1534--1335]{HS14} (and before that by Alesker \cite{Ale99b}), it is sufficient to prove the assertion only for the case where $\Gamma$ is translation invariant. We sketch the general idea of this reduction, because it will be used more than once. It is based on relation (\ref{2.1}), where, as it follows from \cite[Lem. 3.1]{HS14}, the mapping $\Gamma_{p-q}$ is translation invariant. In the cases considered, this is sufficient 
to identify $\Gamma_{p-q}$, in the way that an explicit mapping $\Delta:\Kn\times\B(\Sigma^n)$ can be found such that
$$ \Delta(K+t,\eta+t) = \sum_{j=0}^q \Delta_{p-j}(K,\eta)\frac{t^j}{j!}$$
with tensors $\Delta_{p-j}(K,\eta)\in \T^{p-j}$, and where $\Delta_{p-q}=\Gamma_{p-q}$. The mapping $\Gamma':=\Gamma-\Delta$ then has properties analogous to those of $\Gamma$ and satisfies
$$\Gamma'(K+t,\eta+t) =\sum_{j=0}^{q-1}\Gamma'_{p-j}(K,\eta)\frac{t^j}{j!} \qquad (\mbox{with }\Gamma'_{p-j}=\Gamma_{p-j}-\Delta_{p-j}).$$
Now $\Gamma'_{p-q+1}$ is translation invariant, and the procedure can be repeated. After finitely many steps, one ends up with an explicit representation of $\Gamma$.

Applying this argument in the present situation, we have to use that
$$ \widetilde\phi^{r,s,j}(P+t,\eta+t) =\sum_{i=0}^r \widetilde\phi^{r-i,s,j}(P,\eta)\binom{r}{i}t^i$$
for $P\in \cP^3$, $\eta\in\B(\Sigma^3)$, $t\in\R^3$ and 
$$ \widetilde\phi^{r,s}_k(P+t,\eta+t) =\sum_{i=0}^r \widetilde\phi^{r-i,s}_k(P,\eta)\binom{r}{i}t^i$$
for $P\in \cP^2$, $\eta\in\B(\Sigma^2)$, $t\in \R^2$.

If now translation invariance of $\Gamma$ is assumed, then it follows from Theorem \ref{Thm2.2} (where only mappings $\phi_k^{r,s,j}$ and $\widetilde \phi_k^{r,s,0}$ with $r=0$ appear) that the restriction of $\Gamma$ to ${\mathcal P}^3\times\B(\Sigma^3)$ is a sum of mappings with the same properties which are homogeneous of one of the degrees $0,1,2$. Therefore, it is sufficient to prove Theorem \ref{Thm4.1} under the additional assumption that $\Gamma$ is homogeneous of degree $k$, for some $k\in\{0,1,2\}$. Since the mappings $\widetilde\phi^{r,s,j}$ in Theorem \ref{Thm2.2} appear only for homogeneity degree one (and linear independence has been proved in \cite{HS16}), it is finally clear that in order to complete the proof of Theorem \ref{Thm4.1}, we only have to prove the following result.

\begin{theorem}\label{Thm4.2}
Let $p\in \N_0$. Let $\Gamma: {\mathcal K}^3\times \B(\Sigma^3)\to \T^p$ be a mapping with the following properties.\\[1mm]
$\rm (a)$ \em $\Gamma(K,\cdot)$ is a $\T^p$-valued measure, for each $K\in {\mathcal K}^3$,\\[1mm]
$\rm (b)$ \em $\Gamma$ is translation invariant and ${\rm SO}(3)$ covariant,\\[1mm]
$\rm (c)$ \em $\Gamma$ is locally defined,\\[1mm]
$\rm (d)$ \em $\Gamma$ is weakly continuous,\\[1mm]
$\rm (e)$ \em $\Gamma$ is homogeneous of degree $1$.\\[1mm]
Then $\Gamma$ is a linear combination, with constant coefficients, of the mappings $Q^m\phi_1^{0,s,j}$, where $m,s\in\N_0$ and $j\in\{0,1\}$ satisfy $2m+2j+s=p$, and of the mappings $Q^m\widetilde\phi^{0,s,0}$, where $m,s\in\N_0$ satisfy $2m+s+2=p$.
\end{theorem}

To begin with the proof, let $\Gamma$ be a mapping satisfying the assumptions (a)--(e) of Theorem \ref{Thm4.2}. By Theorem \ref{Thm2.2}, on polytopes $P$ the mapping $\Gamma$ is of the form
$$ \Gamma(P,\cdot) =\sum_{m,j,s\ge 0\atop 2m+2j+s=p} c_{mjs}Q^m\phi_1^{0,s,j}(P,\cdot) + \sum_{m,j,s\ge 0\atop 2m+2j+s+2=p} a_{mjs}Q^m \widetilde \phi^{0,s,j}(P,\cdot).$$
Since $\phi_1^{0,s,0}$, $\phi_1^{0,s,1}$ and $\widetilde \phi^{0,s,0}$ are defined on ${\mathcal K}^3$ and are weakly continuous, the mapping
\begin{equation}\label{4.3} 
\Gamma' := \Gamma- \sum_{m,j,s\ge 0,\, j\le 1\atop 2m+2j+s=p}  c_{mjs}Q^m\phi_1^{0,s,j} - \sum_{m,s\ge 0\atop 2m+s+2=p} a_{m0s}Q^m \widetilde \phi^{0,s,0}
\end{equation}
has again properties (a)--(e) of Theorem \ref{Thm4.2}. 

Thus, given a mapping $\Gamma': {\mathcal K}^3\times \B(\Sigma^3)\to \T^p$ which has  properties (a)--(e) and which on polytopes $P$ is of the form
$$ \Gamma'(P,\cdot) =\sum_{m,j,s\ge 0,\,j\ge 2\atop 2m+2j+s=p} c_{mjs}Q^m\phi_1^{0,s,j}(P,\cdot) + \sum_{m,s\ge 0,\, j\ge 1\atop 2m+2j+s+2=p} a_{mjs}Q^m \widetilde \phi^{0,s,j}(P,\cdot),$$
we have to show that here all coefficients $c_{mjs},a_{mjs}$ are zero. The principal idea to prove this is similar to the one in \cite{HS14}: if not all coefficients are zero, then weak continuity and ${\rm SO}(3)$-covariance finally lead to a contradiction. The details are partially more subtle. For the proof, we first construct a sequence of polytopes that converges to a convex body $K$ of revolution. If $\Gamma'$ is not identically zero, then it can finally be shown that $\Gamma'$ is not covariant under all proper rotations mapping $K$ into itself, which is the desired contradiction.

\section{The Approximating Polytopes}\label{sec5}

We construct polytopes $P^N_{h,t}$, where $N$ is either $2$ or an odd integer $\ge 3$, and where $h,t>0$. Two variants of this construction are described in \cite{HS14}, pp. 1550--1551 and  pp. 1558--1559. We briefly recall the definition in the special case $n=3$ needed here.

We choose an orthonormal basis $(e_1,e_2,e_3)$ of ${\mathbb R}^3$ and identify the subspace spanned by $e_1,e_2$ with ${\mathbb R}^2$. We denote by ${\rm SO}(3,e_3) \subset {\rm SO}(3)$ the subgroup of all rotations fixing $e_3$.

Starting point is a tessellation of ${\mathbb R}^2$ into squares or triangles. Together with all their faces, the polygons of the tessellation form a polygonal complex, which we denote by ${\mathcal C}$. For $t>0$, we denote by $t{\mathcal C}$ the complex obtained from ${\mathcal C}$ by dilatation with the factor $t$. 

With the lifting map $L: {\mathbb R}^2\to{\mathbb R}^3$ defined by
$$ L(x):= x+\|x\|^2e_3,\quad x\in{\mathbb R}^2,$$
we define the polyhedral set 
$$ R(t{\mathcal C}):={\rm conv}\,L({\rm vert}\,t{\mathcal C}),$$
where ${\rm vert}\,t{\mathcal C}$ denotes the set of vertices of the complex $t{\mathcal C}$, and the convex set
$$ K:= {\rm conv}\,L({\mathbb R}^2),$$
which is bounded by a paraboloid of revolution. 

Let $\Pi:{\mathbb R}^3\to {\mathbb R}^2$ denote the orthogonal projection. It is a well-known fact (for references, see \cite[pp. 1550--1551]{HS14}) that each face $F$ of $R(t{\mathcal C})$ {\em lies above} a face $G$ of $t{\mathcal C}$ of the same dimension, in the sense that $\Pi F=G$. We write $G= F^\Box$.

We describe the special complexes ${\mathcal C}$ that we use. For $N=2$, it is the complex ${\mathcal C}_2$ of unit squares and their faces, defined by the vertex set $ \{(m_1,m_2,0) \in{\mathbb R}^3: m_1,m_2\in{\mathbb Z}\}$. We point out the trivial (but crucial) fact that each edge of ${\mathcal C}_2$ is parallel to one of the vectors $e_1,e_2$. For the sake of uniformity with later notation, we write $\ell_1=e_1$ and $\ell_2=e_2$.

Now let $N\ge 3$ be an odd integer. We define $\beta_N:= \pi/N$ and the vectors
$$ z_1=e_1,\qquad z_2=(\cos\beta_N)e_1+(\sin\beta_N)e_2,\qquad z_3=z_2-z_1.$$
The triangle $T$ with vertices $0$, $z_1$ and $z_2$ has angles $\beta_N$ at $0$ and $((N-1)/2)\beta_N$ at $z_1$ and at $z_2$. The lines ${\mathbb R}z_1+mz_2$, ${\mathbb R}z_2+mz_3$, ${\mathbb R}z_3+mz_1$ with $m\in{\mathbb Z}$ tessellate the plane ${\mathbb R}^2$ into triangles which are translates of $T$ or $-T$. Together with their faces, they form the polygonal complex ${\mathcal C}_N$. We define the vectors
\begin{equation}\label{5.1a} 
\ell_r:= (\cos r\beta_N)e_1+(\sin r\beta_N)e_2,\qquad r=0,1,\dots,N-1,
\end{equation}
in ${\mathbb R}^2$ and observe that each edge of the complex ${\mathcal C}_N$ is parallel to one of the vectors $\ell_0,\ell_1,\ell_{(N+1)/2}$. (Here we use that $N\ge 3$ is odd.) 

Now we define convex polytopes. Let $h>0$. We cut the polyhedral set $R(t{\mathcal C}_2)$ by the closed halfspace
$$ H_h^-:=\{y\in {\mathbb R}^3: \langle y,e^3\rangle\le h\}$$
and define
$$ P^2_{h,t}:= R(t{\mathcal C}_2)\cap H_h^-.$$
For an odd integer $N\ge 3$, let $\vartheta_N \in{\rm SO}(3, e_3)$ denote the rotation by the angle $\beta_N$ that fixes $e_3$. Then we define the Minkowski average
$$ P^N_{h,t}:= \frac{1}{N} \sum_{k=0}^{N-1} \vartheta_N^k (R(t{\mathcal C}_N)\cap H_h^-).$$
These polytopes satisfy
\begin{equation}\label{5.1} 
\vartheta_N P^N_{h,t} = P^N_{h,t}.
\end{equation}
This holds also for $N=2$, if $\vartheta_2$ denotes the rotation by the angle $\pi/2$ that fixes $e_3$.

Defining the convex body
\begin{equation}\label{5.2}  
K_h:= K \cap H_h^-,
\end{equation}
we clearly have 
\begin{equation} \label{5.3}
\lim_{t\to 0} P^N_{h,t} =K_h
\end{equation}
in the Hausdorff metric, for all $N\in\{2,3,5,7,\dots\}$.

In the following, the edges of $P^N_{h,t}$ that are edges of $R(t{\mathcal C}_N)$ will play a particular role. As proved in \cite{HS14}, they belong to $N$ different classes, which we denote by
$$ {\mathcal E}_r(P^N_{h,t}):= \{ F\in{\mathcal F}_1(P^N_{h,t}): \Pi F \mbox{ is an edge of $t{\mathcal C}_N$, parallel to } \ell_r\},\quad r=0,\dots,N-1.$$

By $\omega_h$ we denote the set of outer unit normal vectors of the convex body $K_h$ at points in the interior of the halfspace $H^-_{h/2}.$

\section{Satisfying some Assumptions}\label{sec6}

Let $f$ be a continuous real function on ${\mathbb S}^2$ with $0\le f\le 1$ and
$${\rm supp}\,f\subset \omega_h,$$ 
which is not identically zero and is invariant under ${\rm SO}(3,e_3)$. We define the edge set
$$ {\mathcal F}_f(P^N_{h,t}):= \{ F\in{\mathcal F}_1(P^N_{h,t}): \nu(P^N_{h,t},F)\cap{\rm supp}\,f\not=\emptyset\}.$$
For $\varepsilon>0$, we formulate the following assumptions $A(\varepsilon)$, $B(\varepsilon)$, $C(\varepsilon)$  on $f$ and $P^N_{h,t}$.

\vspace{3mm}

\noindent{\bf Assumption $A(\varepsilon)$}: For every $u\in{\rm supp}\,f$,
\begin{equation}\label{6.1}
\langle u,-e_3\rangle >1-\varepsilon
\end{equation}
and
\begin{equation}\label{6.2}
|\langle u,a\rangle|\le\varepsilon \quad\mbox{for }a\in{\mathbb R}^2 \mbox{ with }\|a\|=1.
\end{equation}

\noindent{\bf Assumption $B(\varepsilon)$}:
\begin{equation}\label{6.0} 
{\mathcal F}_f(P^N_{h,t}) \subset \bigcup_{r=0}^{N-1} {\mathcal E}_r(P^N_{h,t}).
\end{equation}

The next assumption uses the vectors $\ell_r$ defined by (\ref{5.1a}). Recall that $v_F$ is one of the unit vectors parallel to the edge $F$. For $F\in {\mathcal E}_r(P^N_{h,t})$, we always have $\langle v_F,\ell_r\rangle\not=0$, and we choose $v_F$ such that $\langle v_F,\ell_r\rangle>0$. We then call $v_F$ the {\em canonical} unit vector for $F$.

\vspace{3mm}

\noindent{\bf Assumption $C(\varepsilon)$}: If $F\in {\mathcal F}_f(P^N_{h,t})\cap {\mathcal E}_r(P^N_{h,t})$ and if $v_F$ is the canonical unit vector for $F$, then
\begin{equation}\label{6.4}
|\langle v_F,a\rangle-\langle \ell_r,a\rangle| \le\varepsilon \quad\mbox{for }a\in{\mathbb R}^2\mbox{ with }\|a\|=1,
\end{equation}
and 
\begin{equation}\label{6.3}
|(v_F \times u)(a) - (\ell_r \times -e_3)(a)| \le\varepsilon\quad\mbox{for }a\in{\mathbb R}^2\mbox{ with }\|a\|=1.
\end{equation}

\vspace{2mm}

\begin{proposition}
For given $\varepsilon>0$, the parameter $h>0$ and a number $\tau>0$ can be chosen such that assumptions $A(\varepsilon), B(\varepsilon), C(\varepsilon)$ are satisfied for $0<t<\tau$.
\end{proposition}

\begin{proof} 
Let $u\in {\rm supp}\,f$, then $u\in\omega_h$ and hence
$$ \langle u,-e_3 \rangle \ge \sqrt{\frac{1}{1+2h}}$$
and 
$$ |\langle u,a \rangle|\le \sqrt{\frac{2h}{1+2h}}\quad\mbox{for $a\in{\mathbb R}^2$ with } \|a\|=1.$$
Therefore, $h>0$ can be chosen such that assumption $A(\varepsilon)$ is satisfied.

To show that assumption $B(\varepsilon)$ can be satisfied, let $F\in{\mathcal F}_f(P^N_{h,t})$, then $\nu(P^N_{h,t},F) \cap {\rm supp}\,f\not=\emptyset$, hence we can choose $u\in \nu(P^N_{h,t},F)$ with $u\in\omega_h$. Let $H(K,u)$ denote the supporting plane of the convex set $K$ with outer normal vector $u$. Let $H_h$ be the boundary plane of the halfspace $H^-_h$. There is a number $\delta>0$ such that each plane $H(K,u)$ with $u\in\omega_h$ has distance at least $\delta$ from the set $K\cap H_h=K_h\cap H_h$. Therefore, there is a number $\tau>0$ (depending on $\varepsilon$) such that for $0<t<\tau$ each supporting plane $H(P^N_{h,t},u)$ with $u\in\omega_h$ has distance at least $\delta/2$ from $K_h\cap H_h$. For these $t$, the edge $F$ cannot contain a point of $H_h$ and hence must be an edge of 
$R(t{\mathcal C})$ lying above some edge $F^\Box$ of $t{\mathcal C}_N$. Hence, if we assume that $0<t<\tau$, then  $F\in {\mathcal E}_r(P^N_{h,t})$ for some $r\in\{0,\dots,N-1\}$. Thus, assumption $B(\varepsilon)$ is satisfied if $0<t<\tau$.

Let $F\in{\mathcal F}_f(P^N_{h,t})$. Then $F\in {\mathcal E}_r(P^N_{h,t})$ for some $r$, by (\ref{6.0}). It is clear that (\ref{6.4}) and (\ref{6.3}) hold if $h>0$ has been chosen sufficiently small. 
\end{proof}

\section{Completing the Proof of Theorem \ref{Thm4.2}}\label{sec7}

Recall that we are given a mapping $\Gamma': {\mathcal K}^3\times \B(\Sigma^3)\to \T^p$ which has  properties (a)--(e) of  Theorem \ref{Thm4.2} and which on polytopes $P$ is of the form
\begin{equation}\label{1.2} 
\Gamma'(P,\cdot)=\sum_{m,s\ge 0, \,j\ge 2\atop 2m+2j+s=p} c_{mjs}Q^m\phi_1^{0,s,j}(P,\cdot) + \sum_{m,s\ge 0,\, j\ge 1\atop 2m+2j+s+2=p} a_{mjs}Q^m \widetilde \phi^{0,s,j}(P,\cdot).
\end{equation} 
We assume that here not all coefficients $c_{mjs},a_{mjs}$ are zero, and we want to reach a contradiction. 

The number $N\in\{2,3,5,7,\dots\}$ will be chosen later, in a way that depends only on the coefficients $c_{mjs}, a_{mjs}$ and thus only on $\Gamma'$. As long as no specific $N$ has been chosen, the assertions involving $N$ hold for any $N\in\{2,3,5,7,\dots\}$.

With the function $f$ chosen in the previous section, we define for $K\in{\mathcal K}^3$ the tensor
$$ \Gamma'(K,f):= \int_{\Sigma^3} f(u)\,\Gamma'(K,\D(x,u)).$$
We also define
$$ W_1(K,f) := 2\pi\int_{{\mathbb S}^2} f(u)\,\Psi_1(K,\D u),$$
where $\Psi_1(K,\cdot)$ is the first order area measure of $K$ (see \cite{Sch14}, Section 4.2). In particular, for $P\in{\mathcal P}^3$, 
$$ W_1(P,f) = \sum_{F\in{\mathcal F}_1(P)} \Ha^1(F) \int_{\nu(P,F)}f\,\D\Ha^1.$$

We will apply $\Gamma'(K,f)$ to arguments of the form
\begin{equation}\label{6.5} 
E_i=(b_1,\dots,b_p)=(\underbrace{a,\dots,a}_{p-i},\underbrace{-e_3,\dots,-e_3}_{i})\quad\mbox{with } a\in {\mathbb R}^2, \,\|a\|=1,
\end{equation}
where $1\le i\le p$. We write
\begin{equation}\label{6.6} 
E_i':= (\underbrace{a,\dots,a}_{p-i}).
\end{equation}

In the following, we say that a symmetric tensor $T$ on $\R^3$ is {\em ${\rm SO}(3,e_3)$ invariant} if
$$ T(\vartheta E_i') = T(E_i') \quad \mbox{for all }\vartheta\in {\rm SO}(3,e_3),$$
for $E_i'$ given by (\ref{6.6}) with $a\in\R^2$, and $\vartheta E_i':= (\vartheta a,\dots,\vartheta a)$.

\begin{lemma}\label{Lem1}
There exist a number $i\in\{1,\dots,p\}$ and a tensor $\Upsilon\in\T^{p-i}$, which depends only on $\Gamma'$ and is not ${\rm SO}(3,e_3)$ invariant, such that the following holds.

Let $N\in\{2,3,5,7,\dots\}$. For given $\varepsilon>0$, let $h,\tau>0$ be chosen such that assumptions $A(\varepsilon),B(\varepsilon),C(\varepsilon)$ are satisfied for $0<t<\tau$, and let $0<t<\tau$. Then, for all $E_i, E_i'$ according to $(\ref{6.5}), (\ref{6.6})$,
$$ \Gamma'(P^N_{h,t},f)(E_i) = N^{-1} W_1(P^N_{h,t},f)\Upsilon(E_i') + R(E_i')$$
with
$$ |R(E_i')| \le C W_1(P^N_{h,t},f)\varepsilon,$$
where $C$ is a constant depending only on $\Gamma'$.
\end{lemma}

When Lemma \ref{Lem1} has been proved, the proof of Theorem \ref{Thm4.2} can be completed as follows. Since the tensor $\Upsilon\in\T^{p-i}$ is not ${\rm SO}(3,e_3)$ invariant, there are an argument $E_i'$ and a rotation $\vartheta \in {\rm SO}(3,e_3)$ such that
$$ |\Upsilon(\vartheta E_i')-\Upsilon(E_i')|=:M>0,$$
where the constant $M$ depends only on $\Gamma'$. Lemma \ref{Lem1} with $E_i:= (E_i',-e_3,\dots,-e_3)$ gives
$$ \Gamma'(P^N_{h,t},f)(E_i) = N^{-1} W_1(P^N_{h,t},f)\Upsilon(E_i') + R(E_i')$$
and a similar relation for $\vartheta E_i'$, hence
\begin{align*}
& |\Gamma'(P^N_{h,t},f)(\vartheta E_i) - \Gamma'(P^N_{h,t},f)(E_i)|\\
&= |N^{-1} W_1(P^N_{h,t},f)\Upsilon(\vartheta E_i') +R(\vartheta E_i') - N^{-1} W_1(P^N_{h,t},f)\Upsilon(E_i') -R(E_i')|\\
& \ge N^{-1}  W_1(P^N_{h,t},f)|\Upsilon(\vartheta E_i')-\Upsilon(E_i')|- |R(\vartheta E_i')|- |R(E_i')|\\
&\ge N^{-1}   W_1(P^N_{h,t},f)(M-2NC\varepsilon).
\end{align*}
Since the constants $M>0$ and $C$ and the number $N$ are independent of $\varepsilon$, we can choose $\varepsilon>0$ so small that $M-2NC\varepsilon>0$. There is a constant $W>0$ such that $N^{-1}W_1(P^N_{h,t},f)\ge W$ for all sufficiently small $t>0$ (see the proof in \cite{HS14}, p. 1556). Hence, for all sufficiently small $t>0$, we have 
$$ |\Gamma'(P^N_{h,t},f)(\vartheta E_i)-\Gamma'(P^N_{h,t},f)(E_i)| \ge W(M-2NC\varepsilon)>0.$$
From (\ref{5.3}) and the weak continuity of $\Gamma'$ we get
$$ |\Gamma'(K_h,f)(\vartheta E_i) -\Gamma'(K_h,f)(E_i)|\ge W(M-2NC\varepsilon)>0,$$
which, because of $\vartheta K_h=K_h$ and the ${\rm SO}(3,e_3)$ invariance of $f$, contradicts the rotation covariance of $\Gamma'$. This contradiction proves Theorem \ref{Thm4.2}.

\vspace{3mm}

\noindent{\em Proof of  Lemma \ref{Lem1}.}

We may assume that $v_F$, for $F\in {\mathcal E}_r(P^N_{h,t})$, is the canonical unit vector for $F$, and that $A(\varepsilon),B(\varepsilon),C(\varepsilon)$ are satisfied. For $P\in{\mathcal P}^3$ we get from (\ref{1.2}), (\ref{4.1}) and (\ref{2.2}) an explicit representation of $\Gamma'(P,\cdot)$. Changing the notation, replacing $c_{mjs}/s!\omega_{s+2}$ by $c_{mjs}$ (which is irrelevant, since both are $\not=0$) and recalling that $Q_{L(F)}=v_F^2$, we get, for $E_i$ according to (\ref{6.5}),
\begin{align}\label{6.7}
& \Gamma'(P,f)(E_i)  \nonumber\\
& =\sum_{m,s\ge 0, \,j\ge 2\atop 2m+2j+s=p} c_{mjs}  \sum_{F\in{\mathcal F}_1(P)}  \Ha^1(F) \int_{\nu(P,F)} \left(Q^m v_F^{2j}u^s\right)(E_i) f(u)\,\Ha^1(\D u) \nonumber\\
& \hspace{4mm} + \sum_{m,s\ge 0,\, j\ge 1\atop 2m+2j+s+2=p} a_{mjs} \sum_{F\in{\mathcal F}_1(P)}  \Ha^1(F) \int_{\nu(P,F)} \left(Q^m v_F^{2j+1}(v_F\times u)u^s\right)(E_i)f(u)\,\Ha^1(\D u). 
\end{align}
When we apply this to $P=P^N_{h,t}$, we have to observe that
$$ \int_{\nu(P^N_{h,t},F)} \left(Q^m v_F^{2j}u^s\right)(E_i) f(u)\,\Ha^1(\D u) \not=0$$
implies $\nu(P^N_{h,t},F)\cap{\rm supp}\,f\not=\emptyset$, hence $F\in{\mathcal F}_f(P^N_{h,t})$, and then (\ref{6.0}) yields $F\in{\mathcal E}_r(P^N_{h,t})$ for some $r\in \{0,\dots,N-1\}$. A similar observation concerns the second sum in (\ref{6.7}). Therefore, in the following, only those edges of $P^N_{h,t}$ need to be taken into account which belong to ${\mathcal E}_r(P^N_{h,t})$ for some $r$. Thus, we get
\begin{equation}\label{6.14}
\Gamma'(P^N_{h,t},f)(E_i) =\sum_{m,s\ge 0, \,j\ge 2\atop 2m+2j+s=p} c_{mjs}\,S_{mjs}(E_i) + \sum_{m,s\ge 0,\, j\ge 1\atop 2m+2j+s+2=p} a_{mjs} \,T_{mjs}(E_i)
\end{equation}
with
\begin{align}
S_{mjs}(E_i) &= \sum_{r=0}^{N-1} \sum_{F\in{\mathcal E}_r(P^N_{h,t})}  \Ha^1(F) \int_{\nu(P^N_{h,t},F)} \left(Q^m v_F^{2j}u^s\right)(E_i) f(u)\,\Ha^1(\D u), \label{6.15}\\
T_{mjs}(E_i) &=\sum_{r=0}^{N-1} \sum_{F\in{\mathcal E}_r(P^N_{h,t})}  \Ha^1(F) \int_{\nu(P^N_{h,t},F)} \left(Q^m v_F^{2j+1}(v_F\times u)u^s\right)(E_i)f(u)\,\Ha^1(\D u). \label{6.16}
\end{align}
According to the definition of the symmetric tensor product, the terms appearing here in the integrands are explicitly given by
\begin{align}\label{6.20}
& \left(Q^m v_F^{2j} u^s\right)(E_i) \nonumber\\
&= \frac{1}{p!} \sum_{\sigma\in {\mathcal S}(p)} Q^m(b_{\sigma(1)},\dots, b_{\sigma(2m)})v_F^{2j}(b_{\sigma(2m+1)}, \dots,b_{\sigma(2m+2j)})\nonumber\\
&  \hspace{4mm}\times u^s(b_{\sigma(2m+2j+1)},\dots,b_{\sigma(p)}) 
\end{align}
and
\begin{align}\label{6.21}
& \left(Q^m v_F^{2j+1} (v_F\times u) u^s\right)(E_i)\nonumber \\
&= \frac{1}{p!} \sum_{\sigma\in {\mathcal S}(p)} Q^m(b_{\sigma(1)},\dots, b_{\sigma(2m)}) v_F^{2j+1}(b_{\sigma(2m+1)},\dots,b_{\sigma(2m+2j+1)})\nonumber\\
& \hspace{4mm}\times  (v_F\times u)(b_{\sigma(2m+2j+2)}) 
u^s(b_{\sigma(2m+2j+3)},\dots,b_{\sigma(p)}). 
\end{align}

Since the components of $E_i=(b_1,\dots,b_p)=(a,\dots,a,-e_3,\dots,-e_3)$ are unit vectors, we have
\begin{equation}\label{6.8} 
|Q^m(\cdot)|\le 1,\quad |v_F^{2j}(\cdot)|\le 1, \quad |v_F^{2j+1}(\cdot)|\le 1, \quad |u^s(\cdot)|\le 1, \quad |(v_F\times u)(\cdot)|\le 1.
\end{equation}
If at least one argument of $u^s$ is equal to $a$, then it follows from (\ref{6.2}) that
\begin{equation}\label{6.9}   
|u^s(\cdot)|\le\varepsilon.
\end{equation}
We have $u^s(-e_3,\dots,-e_3)= \langle u,-e_3\rangle^s$ and, by (\ref{6.1}), $\langle u,-e_3\rangle \ge 1-\varepsilon$, hence
\begin{equation}\label{6.10}  
1\ge u^s(-e_3,\dots,-e_3)\ge 1-s\varepsilon.
\end{equation}
The face $F$ which we have to consider belongs to some ${\mathcal E}_r(P^N_{h,t})$, and then (\ref{6.4}) yields
\begin{equation}\label{6.12}  
|v_F^{2j}(a,\dots,a)-\ell^{2j}_r(a,\dots,a)| \le 2j\varepsilon. 
\end{equation}

If not all coefficients $c_{mjs}$ are zero, we denote by $s_0$ the smallest number $s$ for which $c_{mjs}\not=0$ for some $m,j$.  If not all coefficients $a_{mjs}$ are zero, we denote by $t_0$ the smallest number $s$ for which $a_{mjs}\not=0$ for some $m,j$.

Now we have to distinguish several cases. In the following, we denote by $C$ a constant (not always the same) that depends only on $\Gamma'$.

\noindent{\bf Case 1:} Not all $c_{mjs}$ are zero, not all $a_{mjs}$ are zero, and $t_0>s_0$. 

In (\ref{6.5}), we choose $i=s_0$. In the second sum of (\ref{6.14}) we have $s\ge t_0> s_0$ whenever $a_{mjs}\not=0$ for some $m,j$. Hence, in (\ref{6.21}) (for $i=s_0$), each term $u^s(\cdot)$ contains at least one argument equal to $a$. Therefore, it follows from (\ref{6.8}) and (\ref{6.9}) that $|T_{mjs}(E_{s_0})|\le C\varepsilon$. 

By the same argument, we have $|(Q^m v_F^{2j} u^s)(E_{s_0})|\le C\varepsilon$ if $s>s_0$. 

In (\ref{6.20}) for $s=s_0$, each summand in which the term $u^{s_0}(\cdot)$ has at least one argument equal to $a$, has absolute value less than $C\varepsilon$. In the remaining $s_0!(p-s_0)!$ summands, all arguments of $u^{s_0}(\cdot)$ are equal to $-e_3$, and in these summands, we have
$$ |u^{s_0}(\cdot)-1|\le C\varepsilon$$
by (\ref{6.10}) and 
$$ |v_F^{2j}(\cdot)-\ell_r^{2j}(\cdot)|\le C\varepsilon$$
by (\ref{6.12}), with $r$ determined by $F$.

To simplify the estimates, we set 
$$ q:= (p-s_0)/2, \qquad c_j:= \binom{p}{s_0}^{-1}c_{(q-j)js_0}.$$
Let $d$ be the largest $j\in\{2,\dots,q\}$ for which $c_j\not=0$. Further, we set
$$ W_{1,r}(P^N_{h,t},f):= \sum_{F\in {\mathcal E}_r(P^N_{h,t})} \Ha^1(F)\int_{\nu(P^N_{h,t},F)} f\,\D\Ha^1.$$
It follows from (\ref{5.1}) and the rotational symmetry of $f$ that this is the same for all $r$, hence
$$ W_{1,r}(P^N_{h,t},f)= N^{-1}W_1(P^N_{h,t},f).$$

Taking the collected estimates together, we arrive at
$$ \Gamma'(P^N_{h,t},f)(E_{s_0}) = N^{-1} W_1(P^N_{h,t},f) \Upsilon_1(E'_{s_0}) + R_1(E'_{s_0})$$
with
$$ \Upsilon_1= \sum_{j=2}^d c_jQ^{q-j} \sum_{r=0}^{N-1} \ell_r^{2j}$$
and
$$ |R_1(E'_{s_0})|\le CW_1(P^N_{h,t},f)\varepsilon.$$
Here $c_d\not=0$. Since $N$ will later be chosen in dependence of $\Gamma'$ alone, the tensor $\Upsilon_1$ depends only on $\Gamma'$.

\vspace{3mm}

\noindent{\bf Case 2:} All $a_{mjs}$ are zero.

Then not all $c_{mjs}$ are zero. We arrive at the same conclusion as in Case 1.

\vspace{3mm}

\noindent{\bf Case 3:} Not all $c_{mjs}$ are zero, not all $a_{mjs}$ are zero, and $s_0>t_0$. 

In (\ref{6.5}), we choose $i=t_0$. In the first sum of (\ref{6.14}) we have $s\ge s_0> t_0$ whenever $c_{mjs}\not=0$ for some $m,j$. Hence, in (\ref{6.20}) (for $i=t_0$), each term $u^s(\cdot)$ contains at least one argument equal to $a$. Therefore, it follows from (\ref{6.8}) and (\ref{6.9}) that $|S_{mjs}(E_{t_0})|\le C\varepsilon$. 

By the same argument, we have $|(Q^m v_F^{2j+1}(v_F\times u) u^s)(E_{t_0})|\le C\varepsilon$ if $s>t_0$. 

In (\ref{6.21}) for $s=t_0$, each summand in which the term $u^{t_0}(\cdot)$ has at least one argument equal to $a$, has absolute value less than $C\varepsilon$. In the remaining $t_0!(p-t_0)!$ summands, all arguments of $u^{t_0}(\cdot)$ are equal to $-e_3$, and in these summands, we have
$$ |u^{t_0}(\cdot)-1|\le C\varepsilon$$
by (\ref{6.10}), 
$$ |(v_F\times u)(\cdot) -(\ell_r\times -e_3)(\cdot)| \le C\varepsilon$$
by (\ref{6.3}), and
$$ |v_F^{2j}(\cdot)-\ell_r^{2j}(\cdot)|\le C\varepsilon$$
by (\ref{6.12}), with $r$ determined by $F$.

We set, in this case,
$$ q:= (p-t_0)/2, \qquad a_j:= \binom{p}{t_0}^{-1}a_{(q-j-2)jt_0}.$$
Let $b$ be the largest $j\in\{2,\dots,q\}$ for which $a_j\not=0$. Similarly as in Case 1, we obtain
$$ \Gamma'(P^N_{h,t},f)(E_{t_0}) = N^{-1} W_1(P^N_{h,t},f) \Upsilon_2(E'_{t_0}) + R_2(E'_{t_0})$$
with
$$ \Upsilon_2= \sum_{j=1}^b a_j Q^{q-j-2} \sum_{r=0}^{N-1} \ell_r^{2j+1}(\ell_r\times -e_3)$$
and
$$ |R_2(E'_{t_0})|\le CW_1(P^N_{h,t},f)\varepsilon.$$
Here $a_b\not=0$.

\vspace{3mm}

\noindent{\bf Case 4:} All $c_{mjs}$ are zero.

Then not all $a_{mjs}$ are zero. We arrive at the same conclusion as in Case 3.

\vspace{3mm}

\noindent{\bf Case 5:} Not all $c_{mjs}$ are zero, not all $a_{mjs}$ are zero, and $s_0=t_0$. 

In (\ref{6.5}), we choose $i=s_0 \;(=t_0)$. As in Cases 1 and 3, we arrive at
$$ \Gamma'(P^N_{h,t},f)(E_{s_0}) = N^{-1} W_1(P^N_{h,t},f) \Upsilon_3(E'_{s_0}) + R_3(E'_{s_0})$$
with
$$ \Upsilon_3= \sum_{j=2}^d c_jQ^{q-j} \sum_{r=0}^{N-1} \ell_r^{2j} + \sum_{j=1}^b a_j Q^{q-j-2} \sum_{r=0}^{N-1} \ell_r^{2j+1}(\ell_r\times -e_3)$$
and
$$ |R_3(E'_{s_0})|\le CW_1(P^N_{h,t},f)\varepsilon.$$
Here  $c_d\not=0$ and $a_b\not=0$.

\vspace{3mm}

It remains to show that none of the tensors $\Upsilon_1,\Upsilon_2,\Upsilon_3$ is ${\rm SO}(3,e_3)$ invariant. Let $\nu\in\{1,2,3\}$. For $x(\lambda)=(\lambda,\sqrt{1-\lambda^2},0)$ with $\lambda\in[0,1)$, we define the function
$$ F_\nu(\lambda) := \Upsilon_\nu(\underbrace{x(\lambda),\dots,x(\lambda)}_{2q}).$$
If $\Upsilon_\nu$ is ${\rm SO}(3,e_3)$ invariant, then the function $F_\nu$ is constant on $[0,1)$. 

Since $x(\lambda)$ is a unit vector, we get
\begin{align}
F_1(\lambda) &= \sum_{j=2}^d c_j\sum_{r=0}^{N-1} \langle\ell_r,x(\lambda)\rangle^{2j}  \nonumber\\
& = \sum_{j=2}^dc_j\sum_{r=0}^{N-1} \left(\lambda\cos r\beta_N+\sqrt{1-\lambda^2}\sin r\beta_N\right)^{2j}, \label{6.23}\\
F_2(\lambda) &= \sum_{j=1}^b a_j\sum_{r=0}^{N-1} \langle\ell_r,x(\lambda)\rangle^{2j+1}\det(\ell_r,-e_3,x(\lambda)) \nonumber\\
& = \sum_{j=1}^b a_j\sum_{r=0}^{N-1} \left(\lambda\cos r\beta_N+\sqrt{1-\lambda^2}\sin r\beta_N\right)^{2j+1} \left( -\lambda\sin r\beta_N+\sqrt{1-\lambda^2}\,\cos r\beta_N\right)\label{6.24}
\end{align}
and $F_3=F_1+F_2$. 
We have
$$ F_1(\lambda) = P(\lambda)+\sqrt{1-\lambda^2}\,Q(\lambda),\qquad F_2(\lambda) = \widetilde P(\lambda)+ \sqrt{1-\lambda^2} \,\widetilde Q(\lambda)$$
with polynomials $P,Q,\widetilde P,\widetilde Q$ (which are defined for all real $\lambda$). Suppose, for example, that $F_1(\lambda)=c$ with a constant $c$. Then 
$$ (P(\lambda)-c)^2 = (1-\lambda)(1+\lambda) Q(\lambda)^2$$
for all real $\lambda$. If $P-c$ is not identically zero, then $\lambda=1$ is a root of $(P-c)^2$ with odd multiplicity, a contradiction. Therefore, $P=c$ and $Q=0$. Similar assertions hold if $F_2$ or $F_3$ is constant.

Our aim is to show that $F_i$ is not constant for $i=1,2,3$. For this purpose it is convenient to 
 consider $\lambda$ as a complex variable. The polynomials $P,Q,\widetilde P,\widetilde Q$ are defined for all $\lambda\in {\mathbb C}$. The function $\lambda\mapsto \sqrt{1-\lambda^2}$, $\lambda\in (-1,1)$, has a univalent analytic continuation to the complex plane with the set
 $E:=\{a+{\rm i}\, b\in\mathbb{C}: a,b\in\R, ab\ge 0, a^2=1+b^2\}$ removed. More explicitly, we define $\sqrt{z}=\sqrt{r}\, e^{{\rm i}\, \varphi/2}$ for $z=r\, e^{{\rm i}\, \varphi} \in \mathbb{C}\setminus{\rm i}\,\R_{\le 0}$ with $r>0$ and $\varphi\in (-\pi/2,3\pi/2)$ 
and observe that $\lambda \in \mathbb{C}\setminus E$ if and only if $1-\lambda^2\in \mathbb{C}\setminus{\rm i}\,\R_{\le 0}$. This yields  
the analytic continuation $\lambda\mapsto\sqrt{1-\lambda^2}$  for $\lambda\in\mathbb{C}\setminus E$. Moreover, we have  $\sqrt{1-\lambda^2}=\sqrt{\lambda^2-1}\, {\rm i}$ for $\lambda\in \R_{>1}\subset \mathbb{C}\setminus E$. Thus we also obtain the univalent analytic continuation of $F_i$ to the connected domain $\mathbb{C}\setminus E$ for $i=1,2,3$. Since $F_i$ is not constant on $[0,1)$ if 
the analytic continuation of $F_i$ is not constant on $(1,\infty)$, we    
consider limits through real $\lambda\to\infty$. In particular, we obtain 
\begin{align}\label{6.25} 
\lim_{\lambda\to\infty} \frac{F_1(\lambda)}{\lambda^{2d}} = c_d\sum_{r=0}^{N-1} \left(\cos r\beta_N +{\rm i}\sin r\beta_N\right)^{2d} = c_d\sum_{r=0}^{N-1} e^{{\rm i}\cdot 2dr\beta_N}
\end{align}
and hence
\begin{equation}\label{6.26}
\lim_{\lambda\to\infty} \frac{P(\lambda)}{\lambda^{2d}} =c_d\, {\rm Re}\sum_{r=0}^{N-1}  e^{{\rm i}\cdot 2dr\beta_N},\qquad
\lim_{\lambda\to\infty} \frac{Q(\lambda)}{\lambda^{2d-1}} =c_d\, {\rm Im}\sum_{r=0}^{N-1}  e^{{\rm i}\cdot 2dr\beta_N}.
\end{equation}
In a similar way we get
\begin{align} \label{6.27}
\lim_{\lambda\to\infty} \frac{F_2(\lambda)}{\lambda^{2b+2}} &= a_b\sum_{r=0}^{N-1} \left(\cos r\beta_N +{\rm i}\sin r\beta_N\right)^{2b+1} (-\sin r\beta_N+ {\rm i}\cos r\beta_N) \nonumber\\
& = a_b\sum_{r=0}^{N-1} {\rm i}\, e^{{\rm i}\cdot 2(b+1)r\beta_N}
\end{align}
and thus
\begin{equation}\label{6.28}
\lim_{\lambda\to\infty} \frac{\widetilde P(\lambda)}{\lambda^{2b+2}} = -a_b\, {\rm Im}\sum_{r=0}^{N-1}  e^{{\rm i}\cdot 2(b+1)r\beta_N},\qquad
\lim_{\lambda\to\infty} \frac{\widetilde Q(\lambda)}{\lambda^{2b+1}} = a_b\, {\rm Re}\sum_{r=0}^{N-1}  e^{{\rm i}\cdot 2(b+1)r\beta_N}.
\end{equation}

\vspace{3mm}

We show that $\Upsilon_1$ is not ${\rm SO}(3,e_3)$ invariant.

If $d$ is even, we choose $N=2$. Then $F_1$ is a polynomial, and we have
$$ F_1(\lambda) = \sum_{j=2}^d c_j \left(\lambda^{2j}+ (1-\lambda^2)^j\right) = 2c_d\lambda^{2d}+(\mbox{lower order terms}).$$
Since $c_d\not=0$ in Cases 1 and 2, the function $F_1$ is not constant.

If $d$ is odd, we choose $N=d$. Then
$$ \sum_{r=0}^{N-1}  e^{{\rm i}\cdot 2dr\beta_N}=d.$$
From (\ref{6.26}) and $c_d\not=0$ we see that $P$ is not constant, hence $F_1$ is not constant.

We show that $\Upsilon_2$ is not ${\rm SO}(3,e_3)$ invariant.

If $b$ is odd, we choose $N=2$. Then we have
\begin{align*}
F_2(\lambda) &= \sum_{j=1}^b a_j \left(\lambda^{2j+1} \sqrt{1-\lambda^2} -\lambda \sqrt{1-\lambda^2}^{\,2j+1}\right)\\
& = \sqrt{1-\lambda^2}\left(2a_b\lambda^{2b+1}+(\mbox{lower order terms})\right).
\end{align*}
Since $a_b\not=0$ in Cases 3 and 4, the function $F_2$ is not constant.

If $b$ is even, we choose $N=b+1$. Then  $\sum_{r=0}^{N-1}  e^{{\rm i}\cdot 2(b+1)r\beta_N}=b+1$. From (\ref{6.28}) and $a_b\not=0$ we see that $\widetilde Q$ is not identically zero, hence the function $F_2$ is not constant.

To show that $\Upsilon_3$ is not ${\rm SO}(3,e_3)$ invariant, we have to show that the function $F_1+F_2$ is not constant on $[0,1)$. We have
$$F_1(\lambda)+F_2(\lambda)= P(\lambda)+\widetilde P(\lambda) +\sqrt{1-\lambda^2}\left(Q(\lambda)+\widetilde Q(\lambda) \right).$$
If $F_1+F_2$ is constant, then $P+\widetilde P$ is constant and $Q+\widetilde Q$ is identically zero.

We distinguish three cases.

\noindent {\bf Case (i):} $d>b+1$.

If $d$ is even, we choose $N=2$ and get
$$ F_1(\lambda)+F_2(\lambda) = 2c_d\lambda^{2d}+(\mbox{lower order terms}) + \sqrt{1-\lambda^2}\cdot(\mbox{polynomial in $\lambda$}).$$
Since $c_d\not=0$, this is not a constant function.

If $d$ is odd, we choose $N=d$ and obtain from (\ref{6.26}) that $P$ is of degree $2d$ and from (\ref{6.28}) that $\widetilde P$ is of degree at most $2b+2<2d$. Therefore, $P+\widetilde P$ is not constant, hence $F_1+F_2$ is not constant.

\noindent {\bf Case (ii):} $d<b+1$.

If $b$ is odd, we choose $N=2$ and get
$$ F_1(\lambda)+F_2(\lambda) = (\mbox{polynomial in $\lambda$}) + \sqrt{1-\lambda^2}\cdot\left(2a_b\lambda^{2b+1}+(\mbox{lower order terms})\right).$$
Since $a_b\not=0$, this is not a constant function.
 
If $b$ is even, we choose $N=b+1$ and obtain from (\ref{6.28}) that $\widetilde Q$ is of degree $2b+1$ and from (\ref{6.26}) that $Q$ is of degree at most $2d-1<2b+1$. Therefore, $Q+\widetilde Q$ is not the zero polynomial, hence $F_1+F_2$ is not constant.

\noindent {\bf Case (iii):} $d=b+1$.

For even $d$, we choose $N=2$ and get
$$ F_1(\lambda)+F_2(\lambda) = 2c_d\lambda^{2d} +(\mbox{l. o. t.}) +\sqrt{1-\lambda^2}\left(2a_b\lambda^{2b+1} +(\mbox{l. o. t.})\right),$$
which is not a constant function.

For odd $d$, we choose $N=d=b+1$ and see from (\ref{6.26}) that $P$ has degree $2d$. Since $\sum_{r=0}^{N-1} e^{{\rm i}\cdot 2(b+1)r\beta_N}
=N$ is real, (\ref{6.28}) shows that $\widetilde P$ has degree less than $2b+2=2d$. Therefore, $P+\widetilde P$ is not constant, and hence $F_1+F_2$ is not constant.

This completes the proof of Lemma \ref{Lem1} and thus the proof of Theorem \ref{Thm4.2}.

\section{The Global Tensor Valuations in Dimension Three}\label{sec8}

As we have seen, tensor-valued support measures on $\Rn$ with the usual properties that are ${\rm SO}(n)$ covariant but not ${\rm O}(n)$ covariant, exist precisely if $n=2$ or $n=3$. The total (or global) values of these tensor measures (i.e., the measures evaluated at $\Sigma^n$) yield tensor valuations which are covered by Alesker's characterization theorem (extended to ${\rm SO}(n)$-covariance). In this and the next section, we have a closer look at these global tensor valuations.

In the present section, we assume that $n=3$, and we consider the  tensor valuations defined by the total values of the local tensor valuations $\widetilde \phi^{r,s,0}$, that is, the mappings
$$ T_{r,s}:{\mathcal K}^3 \to \T^{r+s+2}(\R^3), \quad T_{r,s}(K)=\widetilde \phi^{r,s,0}(K,\Sigma^3), \quad K\in{\mathcal K}^3,\, r,s\in{\mathbb N}_0.$$
Thus, for a polytope $P\in\cP^3$ we have
\begin{equation}\label{8.1} 
T_{r,s}(P) = \sum_{F\in\F_1(P)}v_F \int_F x^r\,\cH^1(\D x)\int_{\nu(P,F)} (v_F\times u)u^s\,\cH^1(\D u).
\end{equation}
It follows from Section \ref{sec3} that $T_{r,s}$ is continuous on ${\mathcal K}^3$, moreover, it is a translation covariant, ${\rm SO}(3)$ covariant valuation. As Alesker \cite[p. 246]{Ale99b} pointed out after the proof of his characterization theorem, the replacement of ${\rm O}(n)$-covariance by ${\rm SO}(n)$-covariance in his characterization theorem yields no new valuations in dimensions $n\ge 3$. Therefore, $T_{r,s}$ must, in fact, be ${\rm O}(3)$ covariant. On the other hand, if $\vartheta\in{\rm O}(3)$ changes the orientation, then $T_{r,s}(\vartheta K)= -\vartheta T_{r,s}(K)$ if $K$ is a polytope, and by continuity this extends to general convex bodies $K$. It follows that $T_{r,s}=0$. For this fact, which here is a consequence of Alesker's deep characterization theorem, we want to give a direct proof, which does not need representation theory but uses only some known elementary facts about valuations.

\begin{proposition}\label{Prop2}
The tensor valuation $T_{r,s}$ is identically zero, for $ r,s\in{\mathbb N}_0$. 
\end{proposition}

\begin{proof}
First let $P\in\cP^3$ be a polygon, $\dim P=2$. We choose a positively oriented, orthonormal basis $(e_1,e_2,e_3)$ of $\R^3$ and identify the space spanned by $e_1,e_2$ with $\R^2$. The basis $(e_1,e_2)$ induces an orientation of $\R^2$. Without loss of generality, we assume that $P\subset\R^2$. Set
$$ a=\lambda_1e_1+\lambda_2e_2+\lambda_3e_3 \quad\mbox{with } \lambda_i\in\R.$$
We evaluate
$$ T_ {r,s}(P)(\underbrace{a,\dots,a}_{r+s+2})=  \sum_{F\in\F_1(P)}\langle v_F,a\rangle \int_F \langle x,a\rangle^r\,\cH^1(\D x)\int_{\nu(P,F)} \det(v_F,u,a)\langle u,a\rangle^s\, \cH^1(\D u).$$

Let $F\in\F_1(P)$ be a given edge. Let $u_F\in\R^2$ be the outer unit normal vector of the polygon $P$ at its edge $F$. There is a unique angle $\beta_F\in[0,2\pi)$ with
$$ u_F = (\cos\beta_F)e_1+(\sin\beta_F)e_2.$$
We choose the unit vector $v_F$ in such a way that $(u_F,v_F)$ is a positively oriented basis of $\R^2$, then 
$$ v_F= -(\sin\beta_F)e_1+(\cos\beta_F)e_2.$$
For $u\in \nu(P,F)$, there is a unique angle $\alpha\in[-\pi/2,\pi/2]$ such that
\begin{align*}
u &= (\cos\alpha)u_F+(\sin\alpha)e_3\\
&= (\cos\alpha\cos\beta_F)e_1 +(\cos\alpha\sin\beta_F)e_2+ (\sin\alpha)e_3.
\end{align*}
We obtain
\begin{align*}
\det(v_F,u,a) &= -\lambda_3\cos\alpha +(\lambda_1\cos\beta_F+\lambda_2\sin\beta_F)\sin\alpha,\\
&= \langle u_F,a\rangle\sin\alpha-\lambda_3\cos\alpha,\\
\langle u,a\rangle &= (\lambda_1\cos\beta_F +\lambda_2\sin\beta_F)\cos\alpha+\lambda_3\sin\alpha\\
&= \langle u_F,a\rangle\cos\alpha +\lambda_3\sin\alpha.
\end{align*}
Writing, for the moment, $\langle u_F,a\rangle =A$ and $\lambda_3=B$ (for clearer visibility), we see that
$$
\int_{\nu(P,F)} \det(v_F,u,a)\langle u,a\rangle^s\,\cH^1(\D u)
= \int_{-\pi/2}^{\pi/2}(A\sin\alpha-B\cos\alpha)(A\cos\alpha+B\sin\alpha)^s\,\D\alpha.
$$
We choose $\gamma\in[0,2\pi)$ such that 
$A=\sqrt{A^2+B^2}\cos\gamma$ and $B= \sqrt{A^2+B^2}\sin\gamma$, and hence
$$
(A\sin\alpha-B\cos\alpha)(A\cos\alpha+B\sin\alpha)^s=\sqrt{A^2+B^2}^{\,s+1}\sin(\alpha-\gamma)\cos^s(\alpha-\gamma).
$$
Therefore, we get
$$
\int_{\nu(P,F)} \det(v_F,u,a)\langle u,a\rangle^s\,\cH^1(\D u)
=-\sqrt{A^2+B^2}^{\,s+1}\frac{1}{s+1}2\sin^{s+1}\gamma=-\frac{2}{s+1}\lambda_3^{s+1},
$$
which is independent of the edge $F$. 

Let $\rho\in{\rm SO}(2)$ be the unique rotation that maps $u_F$ to $v_F$. In the following, we use the identity
$$ \sum_{F\in\F_1(P)} \langle u_F,t\rangle \int_F x^r\,\cH^1(\D x) = rt\int_P x^{r-1}\,\cH^2(\D x)$$
for $t\in\R^2$, which follows from the divergence theorem; see \cite[(5.106) and (4.3)]{Sch14}. We obtain
\begin{align*}
\sum_{F\in\F_1(P)} \langle v_F,a\rangle \int_Fx^r\, \cH^1(\D x)
& =\sum_{F\in\F_1(P)} \langle u_F,\rho^{-1} a\rangle \int_F x^r\, \cH^1(\D x)\\
& =r(\rho^{-1} a)\int_P x^{r-1}\, \cH^2(\D x)
\end{align*}
and hence
\begin{equation}\label{eqn1} 
\sum_{F\in\F_1(P)} \langle v_F,a\rangle \int_F \langle x,a\rangle^r\,\cH^1(\D x) = r\langle\rho^{-1}a,a\rangle\int_P \langle x,a\rangle^{r-1}\,\cH^2(\D x) =0.
\end{equation} 
Therefore, $T_{r,s}(P)(a,\dots,a)=0$.

Now we write $\varphi_{r,s}(K)= T_{r,s}(K)(a,\dots,a)$ for $K\in{\mathcal K}^3$. Since $P$ above can be an arbitrary two-dimensional polygon and $T_{r,s}$ is continuous, we have $\varphi_{r,s}(K)=0$ for all convex bodies $K$ of dimension $\dim K\le 2$. Thus, $\varphi_{r,s}$ is a real valuation on ${\mathcal K}^3$ which is continuous, simple and homogeneous of degree $r+1$. Its translation behaviour follows from (\ref{8.1}) and continuity, namely
\begin{equation}\label{8.2} 
\varphi_{r,s}(K+t) =\sum_{j=0}^r\binom{r}{j}\varphi_{r-j,s}(K)\langle t,a\rangle^j
\end{equation}
for $K\in{\mathcal K}^3$ and $t\in\R^3$. Generally, if a relation
$$ \varphi_{r,s}(K+t) = \sum_{j=0}^q\psi_j(K)\langle t,a\rangle^j$$
holds for all $K\in{\mathcal K}^3$ and all $t\in\R^3$ (where $a\not=0$), by computing $\varphi_{r,s}(K+t+\overline t)$ in two different ways, we can conclude that $\psi_q(K+t)= \psi_q(K)$, that is, $\psi_q$ is translation invariant. 

Assume, for the moment, that $\varphi_{r,s}$ is translation invariant. From Theorems 6.4.10 and 6.4.13 in \cite{Sch14} (and the fact that $\varphi_{r,s}$ can be written as the sum of an even and an odd valuation), it follows that 
\begin{equation}\label{8.3} 
\varphi_{r,s}(K) = c V_3(K) +\int_{{\mathbb S}^2} g(u)\,S_2(K,\D u) \quad\mbox{for }K\in{\mathcal K}^3,
\end{equation}
where $V_3$ is the volume, $S_2(K,\cdot)$ is the area measure of $K$, $c$ is a constant, and $g:{\mathbb S}^2\to\R$ is an odd continuous function. One can assume that $g$ has no linear part, that is, $\int_{{\mathbb S}^2} g(u)u\, \D u=0$, and then $g$ is uniquely determined (as follows from \cite[Thm.~6.4.9]{Sch14}).

Now $\varphi_{0,s}$ is, in fact, translation invariant. Since $\varphi_{0,s}$ is homogeneous of degree one, whereas the summands in (\ref{8.3}) are homogeneous of degrees three and two, respectively, we conclude that $\varphi_{0,s}=0$. Next, it follows from (\ref{8.2}) and the remark made after it that $\varphi_{1,s}$ is translation invariant. By (\ref{8.3}) and since $\varphi_{1,s}$ is homogeneous of degree two,
$$ \varphi_{1,s}(K)=\int_{{\mathbb S}^2} g(u)\,S_2(K,\D u).$$
Let $\vartheta\in{\rm O}(3)$ be a reflection at a $2$-dimensional subspace containing $a$. Then
$$ T_{1,s}(\vartheta K)(a,\dots,a) = T_{1,s}(\vartheta K) (\vartheta a,\dots,\vartheta a) = -T_{1,s}(K)(a,\dots,a),$$
hence
$$ \int_{{\mathbb S}^2} g(\vartheta u)\,S_2(K,\D u) =  \int_{{\mathbb S}^2} g(u)\,S_2(\vartheta K,\D u) = -\int_{{\mathbb S}^2} g(u)\,S_2(K,\D u).$$
Since this holds for all $K\in{\mathcal K}^3$ and since the continuous function $g$ has no linear part, it follows that $g(\vartheta u)=-g(u)$ for all $u\in{\mathbb S}^2$. This holds for all reflections $\vartheta$ at 2-subspaces containing $a$. From this, it is easy to deduce that $g=0$. Thus, $\varphi_{1,s}=0$. Again by (\ref{8.2}), this implies that $\varphi_{2,s}$ is translation invariant. By (\ref{8.3}) and since $\varphi_{2,s}$ is homogeneous of degree three, $\varphi_{2,s}(K)= cV_3(K)$ for $K\in{\mathcal K}^3$. For $\vartheta\in{\rm O}(3)$ as above, we have $ T_{2,s}(\vartheta K)(a,\dots,a) = T_{2,s}(\vartheta K) (\vartheta a,\dots,\vartheta a) = -T_{2,s}(K)(a,\dots,a)$, hence $cV_3(\vartheta K)=-cV_3(K)$, which gives $c=0$. Therefore, $\varphi_{2,s}=0$. Now it follows that $\varphi_{r,s}=0$ for all $r\ge 3$, since $\varphi_{r,s}$ is homogeneous of degree $r+1>3$.

Thus, for any $K\in{\mathcal K}^3$ we have $T_{r,s}(K)(a,\dots,a)=0$ for all vectors $a\in\R^3$. By multilinearity, this implies that $T_{r,s}(K)=0$. Therefore, $T_{r,s}=0$.
\end{proof}

\section{Linear Dependences in Dimension Two}\label{sec9}

In this section we assume that $n=2$. We consider the total values of the tensor measures (\ref{A}) and (\ref{B}) and write (not caring about normalizations, and using the support measures $\Theta_k$, see \cite[Sec. 4.2]{Sch14}), for $K\in {\mathcal K}^2$ and $k\in\{0,1\}$,
$$ \Psi_k^{r,s}(K) = \int_{\Sigma^2} x^ru^s\,\Theta_k(K,\D(x,u)),$$
together with
$$ \Psi_2^r(K)= \int_K x^r\,\cH^2(\D x),$$
and
\begin{equation}\label{9.0} 
\widetilde \Phi_k^{r,s}(K) =\int_{\Sigma^2} x^r u^s\overline u\,\Theta_k(K,\D(x,u)).
\end{equation}
For $p\in\N_0$, we denote by $T^p_+({\mathcal K}^2)$ the real vector space spanned by all tensor valuations
\begin{equation}\label{9.1}
Q^m\Psi_k^{r,s},\enspace k\in\{0,1\},\, m,r,s\in\N_0,\, 2m+r+s=p, \qquad Q^m\Psi_2^r,\enspace  m,r\in\N_0, \, 2m+r=p,
\end{equation}
and by $T^p_-({\mathcal K}^2)$ the real vector space spanned by all tensor valuations
\begin{equation}\label{9.2}
Q^m\widetilde{\Phi}_k^{r,s}(K),\enspace k\in\{0,1\},\,m,r,s\in\N_0,\, 2m+r+s+1=p.
\end{equation}
It follows from Alesker \cite[Thm. 4.1]{Ale99b} (together with the relation $\overline{u}^2=Q-u^2$ for $u\in \mathbb{S}^1$) that the real vector space of all mappings $\Gamma:{\mathcal K}^2\times \B(\Sigma^2)\to\T^p$ that are translation covariant, ${\rm SO}(2)$ covariant, continuous valuations, is equal to $T^p_+({\mathcal K}^2)+T^p_-({\mathcal K}^2)$. The two subspaces $T^p_+({\mathcal K}^2)$ and $T^p_-({\mathcal K}^2)$ are complementary to each other; this follows by an argument similar to that before Proposition \ref {Prop2}. The linear dependences between the tensor valuations (\ref{9.1}) and the dimension of the vector space $T^p_+({\mathcal K}^2)$ were determined in \cite{HSS08a}. It remains to determine a basis of $T^p_-({\mathcal K}^2)$. The tensor valuations (\ref{9.0}) are, in fact, not linearly independent.

\begin{theorem}\label{Thm9.1}
The tensor valuations $\rm (\ref{9.0})$ satisfy the relations
\begin{equation}\label{9.3}
\widetilde \Phi_1^{r,0}=0,\enspace \widetilde \Phi_0^{0,s} =0\quad\mbox{for } r,s\in\N_0,
\end{equation}
\begin{equation}\label{9.4}
r\widetilde \Phi_1^{r-1,s} + s\widetilde \Phi_0^{r,s-1}=0 \quad\mbox{for } r,s\in\N.
\end{equation}
A basis of $T^p_-({\mathcal K}^2)$ is given by
\begin{equation}\label{9.5}
Q^m \widetilde \Phi_1^{r,s},\enspace m\in\N_0,\,r,s\in\N,\, 2m+r+s+1=p.
\end{equation}
\end{theorem}

\begin{proof}
Let $P\in\mathcal{P}^2$ and  $a\in\R^2$. Then, for $a\in\R^2$,
$$
\widetilde{\Phi}^{r,0}_1(P)(\underbrace{a,\dots,a}_{r+1})=\sum_{F\in\mathcal{F}_1(P)}\langle \overline{u}_F,a\rangle\int_F\langle x,a\rangle^r\, \mathcal{H}^1(\D x)
=0,
$$
as shown in \eqref{eqn1}. By continuity, we thus get $\widetilde{\Phi}^{r,0}_1(K)=0$ for all $K\in \mathcal{K}^2$. Further relations are obtained by applying this to $K+\rho B^2$ with $\rho\ge 0$ and expanding as a polynomial in $\rho$. According to \cite[Thm. 4.2.7]{Sch14}, the measure $\Theta_1(K+\rho B^2,\cdot)$ is the image measure of the measure $\Theta_1(K,\cdot) +\rho\Theta_0(K,\cdot)$ under the mapping $t_\rho:{\rm Nor}\,K \to{\rm Nor}(K+\rho B^2)$, $t_\rho(x,u)= (x+\rho,u)$. Therefore, the transformation formula for integrals gives
\begin{align*}
0 &= \widetilde\Phi_1^{r,0}(K+\rho B^2) = \int_{{\rm Nor}(K+\rho B^2)} x^r\overline u\,\Theta_1(K+\rho B^2,\D(x,u))\\
&= \int_{{\rm Nor}\,K} (x+\rho u)^r\overline u\,\Theta_1(K,\D(x,u)) +\rho \int_{{\rm Nor}\,K} (x+\rho u)^r\overline u\,\Theta_0(K,\D(x,u))\\
&= \sum_{l=0}^r\binom{r}{l} \rho^l\int_{\Sigma^2} x^{r-l}u^l\overline u\,\Theta_1(K,\D(x,u)) + \sum_{l=0}^r\binom{r}{l} \rho^{l+1}\int_{\Sigma^2} x^{r-l}u^l\overline u\,\Theta_0(K,\D(x,u))\\
&= \widetilde \Phi_1^{r,0}(K) + \sum_{j=1}^r \rho^j\left[\binom{r}{j} \widetilde \Phi_1^{r-j,j}(K) + \binom{r}{j-1}\widetilde\Phi_0^{r-j+1,j-1}(K)\right] +\rho^{r+1}\widetilde\Phi_0^{0,r}(K).
\end{align*}
Since this holds for all $\rho\ge 0$, we obtain the equations (\ref{9.3}) and (\ref{9.4}).

These relations show that the valuations (\ref{9.5}) span the space $T^p_-({\mathcal K}^2)$. To show that they form a basis, we assume that there is a non-trivial linear relation between them. Then there must also be such a relation between valuations of the same degree of homogeneity; note that $Q^m \widetilde \Phi_1^{r,s}$ is homogeneous of degree $r+1$. Considering homogeneity of degree $k\in\{1,\dots,p\}$, we assume a linear relation
$$ \sum_{m=0}^{\lfloor\frac{p-k-1}{2}\rfloor} a_m Q^m\int_{\Sigma^2} x^{k-1}u^{p-k-2m}\overline u\,\Theta_1(K,\D(x,u))=0$$
with real constants $a_m$. Replacing $K$ by $K+t$ with $t\in\R^2$, we get
\begin{align*}
0 &=  \sum_{m=0}^{\lfloor\frac{p-k-1}{2}\rfloor} a_m Q^m\int_{\Sigma^2} (x+t)^{k-1}u^{p-k-2m}\overline u\,\Theta_1(K,\D(x,u))\\
&= \sum_{m=0}^{\lfloor\frac{p-k-1}{2}\rfloor} a_m Q^m \sum_{j=0}^{k-1} \binom{k-1}{j} \int_{\Sigma^2} x^{k-1-j}u^{p-k-2m}\overline u\,\Theta_1(K,\D(x,u))\,t^j.
\end{align*}
If a relation $\sum_{j=0}^{k-1} \Gamma_{p-j}t^j=0$ with fixed tensors $\Gamma_{p-j}\in \T^{p-j}$ holds for all $t$, we can conclude that $\Gamma_{p-j}=0$ (stepwise, using that the symmetric tensor algebra has no zero divisors). In particular, this yields
$$ \sum_{m=0}^{\lfloor\frac{p-k-1}{2}\rfloor} a_m Q^m \int_{\Sigma^2}u^{p-k-2m}\overline u\,\Theta_1(K,\D(x,u))=0. $$
This can also be written as 
\begin{equation}\label{9.6}
\int_{\bS^1} \left(\sum_{m=0}^{\lfloor\frac{p-k-1}{2}\rfloor} a_m Q^m u^{p-k-2m}\overline u \right)S_1(K,\D u)=0,
\end{equation}
where $S_1(K,\cdot)$ is the surface area measure in the plane. We apply this tensor to a $(p-k+1)$-tuple $(e,\dots,e)$ with a unit vector $e\in\bS^1$. Since (\ref{9.6}) holds for all $K\in {\mathcal K}^2$, the integrand is then the restriction of a linear function (e.g., see \cite[Theorem 6.4.9]{Sch14}), thus
\begin{equation}\label{9.7}
\sum_{m=0}^{\lfloor\frac{p-k-1}{2}\rfloor} a_m \langle u,e\rangle^{p-k-2m}\langle\overline u,e\rangle = \langle u,c(e)\rangle
\end{equation}
with a vector $c(e)\in\R^2$ depending on $e$. Relation (\ref{9.7}) holds for all $u\in\bS^1$. For fixed $e$, we write
$$ u=(\cos\alpha)e+(\sin\alpha)\overline e \quad \mbox{with}\enspace \alpha\in (-\pi,\pi],\quad c(e)=c_1e+c_2\overline e, $$
and obtain
$$\sum_{m=0}^{\lfloor\frac{p-k-1}{2}\rfloor} a_m (\cos\alpha)^{p-k-2m} (-\sin\alpha)=c_1\cos\alpha+c_2\sin\alpha.$$
Since the left side is an odd function of $\alpha$, we get $c_1=0$.  Then we can conclude that $a_m=0$ for all $m$.
\end{proof}

\noindent Authors' addresses:\\[2mm]
\parbox{6cm}{Daniel Hug\\
Karlsruhe Institute of Technology \\
Department of Mathematics \\
D-76128 Karlsruhe, Germany\\
e-mail: daniel.hug@kit.edu}
\hspace{1.5cm}
\parbox{7.5cm}{Rolf Schneider\\
Albert-Ludwigs-Universit\"at\\
Mathematisches Institut\\
D-79104 Freiburg i. Br., Germany\\
e-mail: rolf.schneider@math.uni-freiburg.de}


\begin{thebibliography}{00}

\bibitem{Ale99a} Alesker, S., Continuous rotation invariant valuations on convex sets. {\em Ann. of Math.} {\bf 149} (1999), 977--1005.

\bibitem{Ale99b} Alesker, S., Description of continuous isometry covariant valuations on convex sets. {\em Geom. Dedicata} {\bf 74} (1999), 241--248.

\bibitem{BDMW02} Beisbart, C., Dahlke, R., Mecke, K., Wagner, H., Vector- and tensor-valued descriptors for spatial patterns. In {\em Morphology of Condensed Matter} (K. Mecke, D. Stoyan, eds), Lecture Notes in Physics {\bf 600}, pp. 238--260, Springer, Berlin, 2002.

\bibitem{Gla97} Glasauer, S., A generalization of intersection formulae of integral geometry. {\em Geom. Dedicata} {\bf 68} (1997), 101--121.

\bibitem{BH14} Bernig, A., Hug, D., Kinematic formulas for tensor valuations. {\em J. Reine Angew. Math.} (to appear) DOI: 10.1515/crelle-2015-0023, arXiv:1402.2750v1

\bibitem{Fed69} Federer, H., {\em Geometric Measure Theory}. Springer, Berlin, 1969.

\bibitem{Had52} Hadwiger, H., Additive Funktionale k-dimensionaler Eik\"orper I. {\em Arch.Math.} {\bf 3} (1952), 470--478.

\bibitem{Had57} Hadwiger, H., \emph{Vorlesungen \"{u}ber Inhalt, Oberfl\"{a}che und Isoperimetrie}. Sprin\-ger, Berlin, 1957.

\bibitem{HS71} Hadwiger, H., Schneider, R., Vektorielle Integralgeometrie, {\em Elem. Math.} \textbf{26} (1971), 49--57.

\bibitem{HHKM14} H\"orrmann, J., Hug, D., Klatt, M., Mecke, K., Minkowski tensor density formulas for Boolean models. {\em Adv. in Appl. Math.} {\bf 55} (2014), 48--85. 

\bibitem{HKS15} Hug, D., Kiderlen, M., Svane, A. M., Voronoi-based estimation of Minkowski tensors from finite point samples. (submitted), arXiv:1511.02394

\bibitem{HS14} Hug, D., Schneider, R., Local tensor valuations. {\em Geom. Funct. Anal.} {\bf 24} (2014), 1516--1564.

\bibitem{HS16} Hug, D., Schneider, R., ${\rm SO}(n)$ covariant local tensor valuations on polytopes. (submitted) arXiv:1605.00954

\bibitem{HSS08a} Hug, D., Schneider, R., Schuster, R., The space of isometry covariant tensor valuations. 
{\em Algebra i Analiz} {\bf 19} (2007), 194--224, {\em St. Petersburg Math. J.} {\bf 19} (2008), 137--158.

\bibitem{HSS08b} Hug, D., Schneider, R., Schuster, R., Integral geometry of tensor valuations. 
{\em Adv. Appl. Math.} {\bf 41} (2008), 482--509.

\bibitem{KJ16} Kiderlen, M., Jensen, E. B. V. (eds), {\em Tensor Valuations and Their Appplications in Stochastic Geometry and Imaging}, Lecture Notes in Mathematics (in preparation).

\bibitem{KKH14} Kousholt, A., Kiderlen, M., Hug, D., Surface tensor estimation from linear sections. arXiv:1404.6907

\bibitem{McM97} McMullen, P., Isometry covariant valuations on convex bodies. {\em Rend. Circ. Mat. Palermo}, Ser. II, Suppl. {\bf 50} (1997), 259--271.

\bibitem{MKSM13} Mickel, W., Kapfer, S. C., Schr\"oder--Turk, G. E., Mecke, K., Shortcoming of the bond orientational order parameters for the analysis of disordered particulate matter. {\em J. Chem. Phys.} {\bf 138}(4):044501 (2013).

\bibitem{Sai16} Saienko, M.,  Tensor-valued valuations and curvature measures in Euclidean spaces. PhD Thesis, University of Frankfurt, 2016.

\bibitem{Sch72} Schneider, R., Kr\"ummungsschwerpunkte konvexer K\"orper, II. {\em Abh. Math. Sem. Univ. Hamburg} \textbf{37} (1972), 204--217.

\bibitem{Sch75} Schneider, R., Kinematische Ber\"{u}hrma{\ss}e f\"{u}r konvexe K\"{o}rper. \textit{Abh. Math. Sem. Univ. Hamburg} \textbf{44} (1975), 12--23.

\bibitem{Sch78} Schneider, R., Curvature measures of convex bodies. {\em Ann. Mat. Pura Appl.} {\bf 116} (1978), 101--134.

\bibitem{Sch00} Schneider, R., Tensor valuations on convex bodies and integral geometry. {\em Rend. Circ. Mat. Palermo}, Ser. II,  Suppl. {\bf 65 } (2000), 295--316. 

\bibitem{Sch13} Schneider, R., Local tensor valuations on convex polytopes. {\em Monatsh. Math.} {\bf 171} (2013), 459--479.

\bibitem{Sch14} Schneider, R., {\em Convex Bodies: The Brunn--Minkowski Theory.} 2nd edn, Encyclopedia of Mathematics and Its Applications {\bf 151}, Cambridge University Press, Cambridge, 2014.

\bibitem{SS02} Schneider, R., Schuster, R., Tensor valuations on convex bodies and integral geometry, II. {\em Rend. Circ. Mat. Palermo}, Ser. II,  Suppl. {\bf 70 } (2002), 295--314. 

\bibitem{SS06} Schneider, R., Schuster, R., Particle orientation from section stereology. {\em Rend. Circ. Mat. Palermo}, Ser. II,  Suppl. {\bf 77 } (2006), 623--633. 

\bibitem{SchT10} Schr\"{o}der--Turk, G. E., Kapfer, S., Breidenbach, B., Beisbart, C., Mecke, K., Tensorial Minkowski functionals and anisotropy measures for planar patterns. {\em J. Microscopy} {\bf 238} (2010), 57--74.

\bibitem{SchT11} Schr\"{o}der--Turk, G. E., Mickel, W., Kapfer, S. C., Klatt, M. A., Schaller, F. M.,  Hoffmann, M. J. F., Kleppmann, N., Armstrong, P., Inayat, A., Hug, D., Reichelsdorfer, M., Peukert, W., Schwieger, W., Mecke, K., Minkowski tensor shape analysis of cellular, granular and porous structures. {\em Advanced Materials, Special Issue: Hierarchical Structures Towards Functionality} {\bf 23} (2011), 2535--2553. 

\bibitem{SchT13} Schr\"{o}der--Turk, G. E., Mickel, W., Kapfer, S. C., Schaller, F. M., Breidenbach, B., Hug, D., Mecke, K., Minkowski tensors of anisotropic spatial structure. {\em New J. of Physics} {\bf 15} (2013), 083028 (38 pp).

\end{thebibliography}
\end{document}